\documentclass[11pt,a4paper]{article}

\usepackage{graphicx,latexsym,euscript,makeidx,color,bm}
\usepackage{amsmath,amsfonts,amssymb,amsthm,thmtools,mathtools,mathrsfs,enumerate}
\usepackage[colorlinks,linkcolor=blue,citecolor=red]{hyperref}
\usepackage{tikz}
\newcommand*\circled[1]{\tikz[baseline=(char.base)]{
            \node[shape=circle,draw,inner sep=1pt] (char) {#1};}}

\usepackage{geometry}
\allowdisplaybreaks[4]
  \def\hA{\widehat{A}} \def\cA{{\cal A}}
  \def\hB{\widehat{B}} 
  \def\hC{\widehat{C}} \def\cC{{\cal C}}
  \def\hD{\widehat{D}} 
\def\dbE{\mathbb{E}}   
\def\dbF{\mathbb{F}} \def\sF{\mathscr{F}}

   \def\cJ{{\cal J}}

\def\dbP{\mathbb{P}}   \def\cP{{\cal P}}
  \def\hQ{\widehat{Q}} \def\cQ{{\cal Q}}
\def\dbR{\mathbb{R}}  \def\hR{\widehat{R}} \def\cR{{\cal R}}
\def\dbS{\mathbb{S}}  \def\hS{\widehat{S}} \def\cS{{\cal S}}
   
 \def\sU{\mathscr{U}}  
 \def\sV{\mathscr{V}}

\def\bA{\bar{A}}    \def\wX{\widetilde{X}}   \def\tu{\ti u}
\def\bB{\bar{B}}    \def\wY{\widetilde{Y}}   \def\tx{\ti x}
\def\bC{\bar{C}}  
\def\bD{\bar{D}}
\def\bQ{\bar{Q}}
\def\bS{\bar{S}}
\def\bR{\bar{R}}

\def\ss{\smallskip}   \def\lt{\left}          
\def\ms{\medskip}     \def\rt{\right}         
\def\h{\hat}          \def\lan{\langle}       
\def\wh{\widehat}     \def\ran{\rangle}       
\def\q{\quad}               
\def\qq{\qquad}             \def\deq{\triangleq}
\def\no{\noindent}    \def\blan{\big\lan}     \def\sc{\scriptscriptstyle}
\def\hp{\hphantom}    \def\bran{\big\ran}     \def\scT{\sc T}
\def\nn{\nonumber}    \def\Blan{\Big\lan\!\!} \def\ds{\displaystyle}
\def\rf{\eqref}       \def\Bran{\!\!\Big\ran} 
\def\cd{\cdot}        \def\({\Big(}           
\def\ti{\tilde}       \def\){\Big)}           \def\les{\leqslant}
   \def\[{\Big[}           \def\ges{\geqslant}
  \def\]{\Big]}           
\def\a{\alpha}       \def\l{\lambda}    \def\D{\varDelta}   \def\vP{\varPi}
\def\b{\beta}                    \def\Up{\varUpsilon}
\def\d{\delta}       \def\th{\theta}       
\def\e{\varepsilon}      \def\L{\varLambda}
\def\f{\varphi}           \def\Om{\varOmega}
       \def\si{\sigma}    \def\Si{\varSigma}
\def\i{\infty}             \def\Th{\varTheta}
\def\be{\begin{equation}}
\def\bel{\begin{equation}\label}
\def\ee{\end{equation}}
\def\bea{\begin{eqnarray}}
\def\eea{\end{eqnarray}}
\def\bt{\begin{theorem}}
\def\et{\end{theorem}}
\def\bc{\begin{corollary}}
\def\ec{\end{corollary}}
\def\bl{\begin{lemma}}
\def\el{\end{lemma}}
\def\bp{\begin{proposition}}
\def\ep{\end{proposition}}
\def\br{\begin{remark}}
\def\er{\end{remark}}
\def\ba{\begin{array}}
\def\ea{\end{array}}
\def\bd{\begin{definition}}
\def\ed{\end{definition}}
\def\5n{\negthinspace \negthinspace \negthinspace \negthinspace \negthinspace }
\def\4n{\negthinspace \negthinspace \negthinspace \negthinspace }
\def\3n{\negthinspace \negthinspace \negthinspace }
\def\2n{\negthinspace \negthinspace }
\def\1n{\negthinspace }

\newtheoremstyle{thry}
{}      
{}      
{\sl}   
{}      
{\bf}   
{.}     
{.5em}  
{}      
\theoremstyle{thry}

\newtheorem{theorem}{Theorem}[section]
\newtheorem{proposition}[theorem]{Proposition}
\newtheorem{corollary}[theorem]{Corollary}
\newtheorem{lemma}[theorem]{Lemma}

\theoremstyle{definition}
\newtheorem{definition}[theorem]{Definition}

\theoremstyle{remark}
\newtheorem{remark}[theorem]{Remark}

\def\punct{}
\newtheoremstyle{dotless}{}{}{\rm}{}{\bf}{\punct}{.5em}{}
\theoremstyle{dotless}

\newtheorem{taggedassumptionx}{}
\newenvironment{taggedassumption}[1]
 {\renewcommand\thetaggedassumptionx{#1}\taggedassumptionx}
 {\endtaggedassumptionx}

\makeatletter
   
   \@addtoreset{equation}{section}
   \newcommand{\setword}[2]{%
   \phantomsection
   #1\def\@currentlabel{\unexpanded{#1}}\label{#2}%
   }
\makeatother

\begin{document}

\title{\bf Turnpike Properties for Mean-Field Linear-Quadratic Optimal Control Problems}
\author{Jingrui Sun\thanks{Department of Mathematics, Southern University of Science and Technology, Shenzhen,
                           518055, China (Email: {\tt sunjr@sustech.edu.cn}).
                           This author is supported by NSFC grant 11901280 and Guangdong Basic and Applied Basic
                           Research Foundation 2021A1515010031.}
\and
Jiongmin Yong\thanks{Department of Mathematics, University of Central Florida, Orlando, FL 32816, USA
                    (Email: {\tt jiongmin.yong@ucf.edu}).
                    This author is supported by NSF grant DMS-1812921.} }

\maketitle

\no{\bf Abstract.}
This paper is concerned with an optimal control problem for a mean-field linear stochastic differential equation with a quadratic functional in the infinite time horizon.
Under suitable conditions, including the stabilizability, the (strong) exponential, integral, and mean-square turnpike
properties for the optimal pair are established.
The keys are to correctly formulate the corresponding static optimization problem and find the equations
determining the correction processes.
These have revealed the main feature of the stochastic problems which are significantly different from
the deterministic version of the theory.

\ms
\no{\bf Key words.}
Strong turnpike property, mean-field, stochastic optimal control, linear-quadratic, static optimization,
stabilizability, Riccati equation.

\ms
\no{\bf AMS 2020 Mathematics Subject Classification.}  49N10, 49N80, 93E15, 93E20.

\section{Introduction}\label{Sec:Intro}

Let $(\Om,\sF,\dbP)$ be a complete probability space on which a standard one-dimensional Brownian motion
$W=\{W(t);\,t\ges 0\}$ is defined and denote by $\dbF=\{\sF_t\}_{t\ges0}$ the usual augmentation of the
natural filtration generated by $W$.
Consider the following controlled linear mean-field stochastic differential equation (SDE, for short)
\begin{equation}\label{TP:state}\left\{\begin{aligned}
dX(t) &= \big\{AX(t) + \bA\dbE[X(t)] + Bu(t) + \bB\dbE[u(t)] + b \big\}dt \\
      &\hp{=\ } +\big\{CX(t) + \bC\dbE[X(t)] + Du(t) + \bD\dbE[u(t)] + \si\big\}dW(t), \q t\ges0, \\
 X(0) &= x,
\end{aligned}\right.\end{equation}
and the quadratic cost functional
\begin{align}\label{TP:cost}
J_{\scT}(x;u(\cd))
&= \dbE\int_0^T\Bigg[\Blan\begin{pmatrix*}[l]Q & \!S^\top \\ S & \!R\end{pmatrix*}\!
                          \begin{pmatrix}X(t) \\ u(t)\end{pmatrix}\!,
                          \begin{pmatrix}X(t) \\ u(t)\end{pmatrix}\Bran
+2\Blan\begin{pmatrix}q \\ r\end{pmatrix}\!,\begin{pmatrix}X(t) \\ u(t)\end{pmatrix}\Bran\nn\\
&\hp{=\ }\qq\q +\Blan\begin{pmatrix*}[l]\bQ & \!\bS^\top \\ \bS & \!\bR\end{pmatrix*}\!
                \begin{pmatrix}\dbE[X(t)] \\ \dbE[u(t)]\end{pmatrix}\!,
                \begin{pmatrix}\dbE[X(t)] \\ \dbE[u(t)]\end{pmatrix}\Bran\Bigg] dt,
\end{align}
where $A,\bA,C,\bC,Q,\bQ\in\dbR^{n\times n}$, $B,\bB,D,\bD\in\dbR^{n\times m}$, $S,\bS\in\dbR^{m\times n}$,
and $R,\bR\in\dbR^{m\times m}$ are constant matrices with $Q,\bQ,R,\bR$ being symmetric;
the superscript $\top$ denotes the transpose of matrices;
$\lan\cd\,,\cd\ran$ denotes the inner product of two vectors;
$b,\si,q\in\dbR^n$ and $r\in\dbR^m$ are constant vectors;
$x\in\dbR^n$ is called an {\it initial state}, and the process $u(\cd)$, called a {\it control},
is selected from the space
\begin{align*}
\sU[0,T] = \lt\{u:[0,T]\times\Om\to\dbR^m \bigm| u~\text{is $\dbF$-progressively measurable,}~\dbE\int_0^T|u(t)|^2dt<\i \rt\}.
\end{align*}
We introduce the following problem.

\ms

{\bf Problem (MFLQ)$_{\scT}$.}
For a given initial state $x\in\dbR^n$, find a control $u_{\scT}(\cd)\in\sU[0,T]$ such that
\begin{align}\label{TP:opt-u}
J_{\scT}(x;u_{\scT}(\cd)) = \inf_{u(\cd)\in\sU[0,T]} J_{\scT}(x;u(\cd)) \equiv V_{\scT}(x).
\end{align}

\ss

The above problem is referred to as a {\it mean-field linear-quadratic (LQ, for short) optimal control problem}
over the finite time horizon $[0,T]$, whose homogenous version (i.e., the case of $b,\si,q=0$, $r=0$)
was initially studied by Yong \cite{Yong2013} (in which the coefficients of the state equation and the
weighting matrices of the cost functional are allowed to be time-dependent) under the standard condition in LQ theory.
The results of Yong \cite{Yong2013} were later generalized by Sun \cite{Sun2017} to the case of uniformly
convex cost functionals with nonhomogenous terms and by Huang--Li--Yong \cite{Huang-Li-Yong2015} to
the infinite time horizon case.
Along this line, many researchers investigated the mean-field LQ problem from various points of view.
For example, Basei--Pham \cite{Basei-Pham2019} studied the mean-field LQ problem using a weak martingale approach;
Li--Li--Yu \cite{Li-Li-Yu2020} introduced the notion of relax compensators to deal with indefinite problems;
and L\"{u} \cite{Lu2020} considered a mean-field LQ problem for stochastic evolution equations.


\ms

For a fixed time horizon, finite or infinite, Problem (MFLQ)$_{\scT}$ has been well studied in recent years,
whose optimal control usually admits a closed-loop representation via the solutions of two related Riccati
differential/algebra equations (see \cite{Sun-Yong2020}). However, the limiting behavior of the optimal pair
as the time horizon tends to infinity has not yet been fully discussed.

\ms

In the case that the diffusion term of the state equation \rf{TP:state} and the mean-field terms
$\dbE[X(t)],\dbE[u(t)]$ are absent, Problem (MFLQ)$_{\scT}$ reduces to a deterministic LQ problem,
denoted by Problem (LQ)$_{\scT}$, for which Porretta--Zuazua \cite{Porretta-Zuazua2013} and
Tr\'{e}lat--Zuazua \cite{Trelat-Zuazua2015} obtained the following result:
The optimal pair exponentially converges in the transient time (as $T\to\i$) to the minimum point
of the corresponding {\it static optimization problem}, under the controllability assumption on the
state equation and the observability assumption on the cost functional.
More precisely, it was shown in \cite{Porretta-Zuazua2013,Trelat-Zuazua2015} that there exist
constants $K,\l>0$, independent of $T$, such that the optimal pair $(X^*_{\scT}(\cd),u^*_{\scT}(\cd))$
of Problem (LQ)$_{\scT}$ satisfies
\begin{align}\label{turnpike-ODE}
|X^*_{\scT}(t)-x^*|+|u^*_{\scT}(t)-u^*| \les K\[e^{-\l t}+e^{-\l(T-t)}\], \q\forall t\in[0,T],
\end{align}
where $(x^*,u^*)$ is the solution to the following static optimization problem:
\begin{equation}\label{static-ODE}\left\{\begin{aligned}
&\text{Minimize}  \q F_0(x,u) \equiv \lan Qx,x\ran + 2\lan Sx,u\ran + \lan Ru,u\ran + 2\lan q,x\ran + 2\lan r,u\ran, \\
&\text{subject to}\q Ax+Bu+b=0.
\end{aligned}\right.\end{equation}
Clearly, \rf{turnpike-ODE} implies that
\begin{align}
|X^*_{\scT}(t)-x^*|+|u^*_{\scT}(t)-u^*|\les2Ke^{-\l\d T}, \q\forall t\in[\d T,(1-\d)T],
\end{align}
for any $\d\in(0,{1\over2})$. This means that in the major part of the time interval $[0,T]$,
the optimal pair $(X^*_{\scT}(\cd),u^*_{\scT}(\cd))$ of Problem (LQ)$_{\scT}$ stays close
to the solution $(x^*,u^*)$ of the static optimization problem \rf{static-ODE}.
Such a phenomenon is very similar to the {\it turnpike} in the US highway system.
This is the main reason of the name ``turnpike property'' was made.
Mathematically, we refer to \rf{turnpike-ODE} as the {\it exponential turnpike property} of
the (deterministic) LQ optimal control problem.

\ms  

The turnpike property was first realized by von Neumann \cite{Neumann1945} for
infinite time horizon deterministic optimal growth problems\footnote{Later, these problems have been formulated
as infinite time horizon optimal control problems.}.
The name ``turnpike'' was first coined by Dorfman--Samuelson--Solow \cite{Dorfman-Samuelson-Solow1958} in 1958. 
See \cite{Carlson-Haurie-Leizarowitz1991} for an excellent survey.
In the past several decades, the turnpike properties have attracted attentions of many researchers as such a
property often gives people an essential picture of the optimal pair without solving it analytically and leads
to a significant simplification in numerical methods for solving such kind of optimal control problems.
A large number of papers have been published for finite and infinite dimensional problems in the context
of discrete-time and continuous-time systems; see, for example, \cite{Carlson-Haurie-Leizarowitz1991,Zaslavski2006,Damm-Grune-Stieler-Worthmann2014,Zuazua2017,
Grune-Guglielmi2018,Lou-Wang2019,Zaslavski2019,Breiten-Pfeiffer2020,Grune-Guglielmi2021,Sakamoto-Zuazua2021,
Faulwasser-Grune2022} and the references therein.

\ms

The above mentioned works focus on deterministic optimal control problems. We now look at the stochastic case.
To begin with, we first assume that
$$ \bar A=\bar C=0, \q \bar B=\bar D=0, \q \begin{pmatrix*}[l]\bar Q&\bar S^\top \\ \bar S&\bar R\end{pmatrix*}=0, $$
i.e., the problem does not involve mean-field terms.
We denote the corresponding optimal control problem by Problem (SLQ)$_{\scT}$.
For such a problem, to discuss the turnpike properties of the optimal pair, the major difficulty is
to correctly formulate the corresponding static optimization problem.
Note that intuitively mimicking the situation of Problem (LQ)$_{\scT}$ will lead to an incorrect
static optimization problem.
Recently, Sun--Wang--Yong \cite{Sun-Wang-Yong2022} investigated the turnpike properties for
Problem (SLQ)$_{\scT}$ and found that the correct static optimization problem takes the following form:
\begin{equation}\label{static-SDE}\left\{\begin{aligned}
&\text{Minimize}  \q F_1(x,u) \equiv F_0(x,u) + \lan P(Cx+Du+\si),Cx+Du+\si\ran, \\
&\text{subject to}\q Ax+Bu+b=0,
\end{aligned}\right.\end{equation}
where $P>0$ is the solution to the following algebraic Riccati equation:
\begin{align*}
& PA + A^\top P + C^\top PC + Q \\
&\hp{PA} -(PB + C^\top PD + S^\top) (R + D^\top PD)^{-1} (B^\top P + D^\top C + S) = 0,
\end{align*}
whose solvability is guaranteed under some mild conditions. It was shown that the {\it expectation} of the optimal pair exhibits similar turnpike properties
as the deterministic case, that is, for some positive constants $K,\l>0$ independent of $T$,
\begin{equation}\label{TP:SWY}
\big|\dbE[X^*_{\scT}(t)-x^*]\big| + \big|\dbE[u^*_{\scT}(t)-u^*]\big| \les K\[e^{-\l t} + e^{-\l(T-t)}\],
\q\forall t\in[0,T].
\end{equation}
It is worth noting that \rf{TP:SWY} only tells us $(\dbE[X^*_{\scT}(\cd)],\dbE[u^*_{\scT}(\cd)])$
converges exponentially in the transient time. In general, however, we could not expect the following:
\begin{equation}\label{TP:SWY*}
\dbE|X^*_{\scT}(t)-x^*| + \dbE|u^*_{\scT}(t)-u^*| \les K\[e^{-\l t} + e^{-\l(T-t)}\], \q\forall t\in[0,T].
\end{equation}
But the behavior of the path $t\mapsto(X^*_{\scT}(t),u^*_{\scT}(t))$ seems to be more important and
it is a crucial step to establish general nonlinear theory.

\ms

In this paper, we shall investigate the turnpike properties for Problem (MFLQ)$_{\scT}$.
The main novelty and contributions of the paper can be briefly summarized as follows.

\ms

$\bullet$ We are concerned with stochastic mean-field LQ problems.
By including the expectations of the state and the control, our Problem (MFLQ)$_{\scT}$ substantially
extends the ones studied previously.
The significance is that we have successfully find the correct formulation of the static optimization
problem for the current case, merely under the stabilizability of the state equation, which covers that
for Problem (SLQ)$_{\scT}$ (without mean-field terms) presented in \cite{Sun-Wang-Yong2022}.

\ms

$\bullet$ As \rf{TP:SWY*} could not be expected in general, it needs to be corrected.
We find the proper replacement of $(x^*,u^*)$. It turns out that there exist stochastic processes
$\bm X^*(\cd)$ and $\bm u^*(\cd)$ independent of $T$, which can be determined explicitly, such that
the optimal pair $(X^*_{\scT}(\cd),u^*_{\scT}(\cd))$ of Problem (MFLQ)$_{\scT}$ satisfies the following
{\it strong turnpike property}:
\begin{equation}\label{TPs}
\dbE|X^*_{\scT}(t)-\bm X^*(t)|^2 + \dbE|u^*_{\scT}(t)-\bm u^*(t)|^2 \les K\[e^{-\l t} + e^{-\l(T-t)}\], \q\forall t\in[0,T],
\end{equation}
for some constants $K,\l>0$ independent of $T$. The processes $\bm X^*(\cd)$ and $\bm u^*(\cd)$ have
time-invariant means $x^*$ and $u^*$, respectively, and $(x^*,u^*)$ is exactly the solution to the
corresponding static optimization problem.

\ms

$\bullet$ We point out that for deterministic LQ problems, the controllability is normally assumed
to establish the turnpike properties; see, for example, \cite{Porretta-Zuazua2013,Trelat-Zuazua2015,Lou-Wang2019}.
We find that it suffices to assume the stabilizability condition for the state equation, which is much
weaker than the controllability condition. As a matter of fact, for controlled stochastic linear SDEs
(with or without mean-field terms), the former is much natural and easy to check than the latter.

\ms

The rest of the paper is organized as follows.
In \autoref{Sec:Pre}, we make some necessary preliminaries by introducing some notation, assumptions,
and a number of basic results that will be needed later.
In \autoref{Sec:main-result}, we introduce the static optimization problem associated with Problem (MFLQ)$_{\scT}$
and state the main results of the paper.
\autoref{Sec:stability} is devoted to the stability of Riccati equations, which plays a central role in
studying the turnpike properties.
In \autoref{Sec:proof}, we give the proofs of the main results stated in \autoref{Sec:main-result}.

\section{Preliminaries}\label{Sec:Pre}

In this section, we introduce the basic notation and assumptions which will be used throughout this paper.
We also collect a number of basic results on stochastic LQ optimal control problems and prove a simple
matrix inequality for later use.

\ms

For notational simplicity, we let
\begin{equation}\label{notation:1}\begin{aligned}
&\hA=A+\bA, \q \hB=B+\bB, \q \hC=C+\bC, \q \hD=D+\bD, \\
&\hQ=Q+\bQ, \q \hS=S+\bS, \q \hR=R+\bR.
\end{aligned}\end{equation}
Denote by $\dbS^n$ the space of symmetric $n\times n$ real matrices. For $P,\vP\in\dbS^n$, we let
\begin{equation}\label{notation:2}\begin{aligned}
\cQ(P) &= PA+A^\top P+C^\top PC+Q, &\q \wh\cQ(P,\vP) &= \vP\hA+\hA^\top\vP+\hC^\top P\hC+\hQ, \\
\cS(P) &= B^\top P+D^\top PC+S,    &\q \wh\cS(P,\vP) &= \hB^\top\vP+\hD^\top P\hC + \hS, \\
\cR(P) &= R+D^\top PD,             &\q \wh\cR(P)     &= \hR+\hD^\top P\hD.
\end{aligned}\end{equation}
For matrices $M,N\in\dbS^n$, we will write $M\ges N$ (respectively, $M>N$) if $M-N$ is positive
semi-definite (respectively, positive definite).
The following  basic assumptions will be imposed throughout the paper.

\begin{taggedassumption}{(A1)}\label{MFTP:A1}
The weighting matrices in \rf{TP:cost} satisfy
\begin{align*}
R,\hR >0,  \q  Q-S^\top R^{-1} S >0,  \q  \hQ-\hS^\top\hR^{-1}\hS >0.
\end{align*}
\end{taggedassumption}

\begin{taggedassumption}{(A2)}\label{MFTP:A2}
The controlled ordinary differential equation (ODE, for short)
\begin{align}\label{ODE-sys}
\dot X(t) = \hA X(t) +\hB u(t), \q t\ges0
\end{align}
is stabilizable, i.e., there exists a $\wh\Th\in\dbR^{m\times n}$ such that all the eigenvalues of
$\hA+\hB\wh\Th$ have negative real parts. In this case, $\h\Th$ is called a {\it stabilizer} of \rf{ODE-sys}. The controlled SDE
\begin{align}\label{SDE-sys}
 dX(t) = [AX(t)+Bu(t)]dt + [CX(t)+Du(t)]dW(t), \q t\ges0
\end{align}
is $L^2$-stabilizable, i.e., there exists a $\Th\in\dbR^{m\times n}$ such that for every initial state $x$,
the solution $X(\cd)$ of
$$\left\{\begin{aligned}
dX(t) &= (A+B\Th)X(t)dt + (C+D\Th)X(t)dW(t), \q t\ges0, \\
 X(0) &= x
\end{aligned}\right.$$
satisfies $\dbE\int_0^\i|X(t)|^2<\i$. In this case, $\Th$ is called a {\it stabilizer} of \rf{SDE-sys}.
\end{taggedassumption}

Let $C([0,T];\dbS^n)$ be the space of $\dbS^n$-valued continuous functions on $[0,T]$.
With the notation \rf{notation:1} and \rf{notation:2}, the differential Riccati equations
associated with Problem (MFLQ)$_{\scT}$ can be respectively written as
\begin{equation}\label{Ric:P}\left\{\begin{aligned}
& \dot P_{\scT}(t) + \cQ(P_{\scT}(t)) - \cS(P_{\scT}(t))^\top\cR(P_{\scT}(t))^{-1}\cS(P_{\scT}(t)) =0, \\
& P_{\scT}(T) = 0,
\end{aligned}\right.\end{equation}
and
\begin{equation}\label{Ric:Pi}\left\{\begin{aligned}
& \dot\vP_{\scT}(t) + \wh\cQ(P_{\scT}(t),\vP_{\scT}(t))
  -\wh\cS(P_{\scT}(t),\vP_{\scT}(t))^\top\wh\cR(P_{\scT}(t))^{-1}\wh\cS(P_{\scT}(t),\vP_{\scT}(t)) =0, \\
& \vP_{\scT}(T) = 0.
\end{aligned}\right.\end{equation}
We present the following result concerning the existence and uniqueness of solutions to \rf{Ric:P} and \rf{Ric:Pi}.
For a proof, we refer to Yong \cite{Yong2013}.

\begin{lemma}\label{lmm:Ric-sol}
Let {\rm\ref{MFTP:A1}} hold. Then for any $T>0$, the differential Riccati equations
\rf{Ric:P} and \rf{Ric:Pi} admit a unique solution pair
$(P_{\scT}(\cd),\vP_{\scT}(\cd))\in C([0,T];\dbS^n)\times C([0,T];\dbS^n)$ satisfying
$P_{\scT}(t),\vP_{\scT}(t)\ges0$ for all $t\in[0,T]$.
\end{lemma}

When considering mean-field LQ optimal control problems over an infinite time horizon,
we encounter {\it algebraic} Riccati equations (AREs, for short) instead of differential ones.
The following result establishes the unique solvability of the associated AREs,
whose proof can be found in \cite{Huang-Li-Yong2015}; see also \cite[Chapter 3]{Sun-Yong2020}.

\begin{lemma}\label{lmm:ARE-sol}
Let {\rm\ref{MFTP:A1}--\ref{MFTP:A2}} hold. Then the AREs
\begin{align}
& \cQ(P) - \cS(P)^\top\cR(P)^{-1}\cS(P) =0, \label{ARE:P} \\
& \wh\cQ(P,\vP)-\wh\cS(P,\vP)^\top\wh\cR(P)^{-1}\wh\cS(P,\vP) =0 \label{ARE:Pi}
\end{align}
admit a unique solution pair $(P,\vP)$ satisfying $P,\vP>0$. Moreover, the matrix
$$\wh\Th \deq -\wh\cR(P)^{-1}\wh\cS(P,\vP) $$
is a stabilizer of system \rf{ODE-sys}, and the matrix
$$\Th \deq  -\cR(P)^{-1}\cS(P) $$
is a stabilizer of system \rf{SDE-sys}.
\end{lemma}

Comparing \rf{ARE:P} with \rf{Ric:P}, and \rf{ARE:Pi} with \rf{Ric:Pi},
we may guess that there are some connections between $P$ and $P_{\scT}(\cd)$,
and between $\vP$ and $\vP_{\scT}(\cd)$.
The following result, proved in \cite{Sun-Wang-Yong2022}, establishes the connection
between $P$ and $P_{\scT}(\cd)$.
In \autoref{Sec:stability}, we will establish the connection between $\vP$ and $\vP_{\scT}(\cd)$.

\begin{lemma}\label{lmm:Si-P:bound}
Let {\rm\ref{MFTP:A1}--\ref{MFTP:A2}} hold.
Let $P_{\scT}(\cd)$ and $P$ be the unique solutions of \rf{Ric:P} and \rf{ARE:P}, respectively.
Then there exist constants $K,\l>0$, independent of $T$, such that
$$ |P_{\scT}(t)-P| \les Ke^{-\l(T-t)}, \q\forall t\in[0,T]. $$
\end{lemma}

We conclude this section with a simple matrix inequality that will be needed later.

\begin{lemma}\label{lmm:matrix}
Let $M\in\dbR^{n\times m}$ and $K\in\dbS^m$. If $K>0$, then
$$ M(K+M^\top M)^{-1}M^\top \les I. $$
\end{lemma}

\begin{proof}
Let $L\in\dbR^{m\times m}$ be an invertible matrix such that $K=L^\top L$. Then
$$ M(K+M^\top M)^{-1}M^\top = ML^{-1}[I+(ML^{-1})^\top (ML^{-1})]^{-1}(ML^{-1})^\top. $$
Thus, without loss of generality we may assume $K=I$. Observe that
$$ M(I+M^\top M)^{-1}M^\top = MM^\top(I+MM^\top)^{-1}. $$
For any $\e>0$, $N_\e \deq \e I+ MM^\top >0$ and hence
$$ N_\e(I+N_\e)^{-1} = (N_\e^{-1}+ I)^{-1} \les I. $$
Letting $\e\to0$ yields the desired result.
\end{proof}

\section{Main Results}\label{Sec:main-result}

We present the main results of the paper in this section.
The rigorous proofs are deferred to the subsequent sections.

\ms

First, let us introduce the static optimization problem associated with Problem (MFLQ)$_{\scT}$.
From the previous section we know (recall \autoref{lmm:ARE-sol} and the notation of \rf{notation:1}
and \rf{notation:2}) that under \ref{MFTP:A1} and \ref{MFTP:A2}, the ARE \rf{ARE:P} has a unique
positive definite solution $P$. Let
\begin{align}\label{def:sV}
\sV &\deq \big\{(x,u)\in\dbR^n\times\dbR^m ~|~ \hA x+\hB u + b = 0\big\},
\end{align}
and define a continuous function $F:\sV\to\dbR$ by
\begin{equation}\label{def:F}\begin{aligned}
F(x,u) &= \lan\hQ x,x\ran + \lan \hR u,u\ran + 2\lan\hS x,u\ran + 2\lan q,x\ran + 2\lan r,u\ran \\
       &\hp{=\ } +\lan P(\hC x+\hD u+\si),\hC x+\hD u+\si\ran.
\end{aligned}\end{equation}
The static optimization problem associated with Problem (MFLQ)$_{\scT}$ can be stated as follows.

\ms

{\bf Problem (O).} Find a pair $(x^*,u^*)\in\sV$ such that
$$ F(x^*,u^*)=\min_{(x,u)\in\sV}F(x,u) \equiv V. $$

\ss

For the solvability of Problem (O), we have the following result.

\begin{proposition}\label{prop:x*u*-unique}
Let {\rm\ref{MFTP:A1}--\ref{MFTP:A2}} hold. Then Problem {\rm(O)} has a unique solution.
Moreover,  $(x^*,u^*)$ is the solution of Problem {\rm(O)} if and only if there exists
a unique $\l^*\in\dbR^n$ such that
\begin{equation}\label{l+mu:exist}\left\{\begin{aligned}
& \hA x^* + \hB u^* + b = 0, \\
& \hA^\top\l^* + \hQ x^* + \hC^\top P(\hC x^*+\hD u^* + \si) + \hS^\top u^* + q =0, \\
& \hB^\top\l^* + \hR u^* + \hD^\top P(\hC x^*+\hD u^* + \si) + \hS x^* + r =0.
\end{aligned}\right.\end{equation}
\end{proposition}

The proof of \autoref{prop:x*u*-unique} is similar to the case without mean-field terms.
We omit it here and  refer the reader to \cite{Sun-Wang-Yong2022} (or \cite{Yong2018}) for details.

\ms

Let $(x^*,u^*)$ be the solution of Problem (O) and $\l^*\in\dbR^n$ the vector in \rf{l+mu:exist}. Set
\begin{align}\label{si*+Th}
\si^* = \hC x^*+ \hD u^* + \si, \q \Th =  -\cR(P)^{-1}\cS(P),
\end{align}
where $P$ is the solution to the ARE \rf{ARE:P}. Let $X^*(\cd)$ be the solution to the SDE
\begin{equation}\label{SDE:X*}\left\{\begin{aligned}
dX^*(t) &= (A+B\Th)X^*(t)dt + \big[(C+D\Th)X^*(t)+\si^*\big]dW(t), \q t\ges0,\\
 X^*(0) &= 0,
\end{aligned}\right.\end{equation}
and define
\begin{align}\label{bmX+bmu}
\bm X^*(t) \deq X^*(t) + x^*, \q \bm u^*(t) \deq \Th X^*(t) + u^*.
\end{align}
Clearly, $\dbE[X^*(t)]=0$. Hence, one has $\dbE[\bm X^*(t)]=x^*$ and $\dbE[\bm u^*(t)]=u^*$ for all $t\ges0$.

\ms

We now state the main result of the paper, which establishes the exponential turnpike
property of Problem (MFLQ)$_{\scT}$.

\begin{theorem}\label{thm:main}
Let {\rm\ref{MFTP:A1}--\ref{MFTP:A2}} hold. Let $(X^*_{\scT}(\cd),u^*_{\scT}(\cd))$ be the
optimal pair of Problem {\rm(MFLQ)$_{\scT}$} for the initial state $x$. Then there exist
constants $K,\l>0$, independent of $T$, such that
\begin{align}\label{X*-X*T}
\dbE|X^*_{\scT}(t)-\bm X^*(t)|^2 + \dbE|u^*_{\scT}(t)-\bm u^*(t)|^2
\les K\[e^{-\l t} + e^{-\l(T-t)}\], \q\forall t\in[0,T].
\end{align}
\end{theorem}

We call \rf{X*-X*T} the {\it strong exponential turnpike property} for the optimal pair $(X^*_{\scT}(\cd)$, $u^*_{\scT}(\cd))$.
The above result has several consequences. The first corollary establishes
the {\it strong integral} and the {\it mean-square turnpike properties} for Problem (MFLQ)$_{\scT}$, whose proof is direct.

\begin{corollary}\label{crlry:jifen-TP}
Let {\rm\ref{MFTP:A1}--\ref{MFTP:A2}} hold. Then,
\begin{equation*}
\lim_{T\to\i}{1\over T}\int_0^T\dbE\[|X^*_{\scT}(t)-\bm X^*(t)|^2+|u^*_{\scT}(t)-\bm u^*(t)|^2\]dt=0,
\end{equation*}
Consequently,
\begin{equation*}
\lim_{T\to\i}{1\over T}\int_0^T\[\big|\dbE[X^*_{\scT}(t)]-x^*\big|^2+\big|\dbE[u^*_{\scT}(t)]-u^*\big|^2\]dt=0.
\end{equation*}
\end{corollary}

The second corollary shows that for any initial state $x$, the value $V_{\scT}(x)$ of Problem (MFLQ)$_{\scT}$
converges to the minimum of Problem (O) in the time-average sense.

\begin{corollary}\label{crlry:value}
Let {\rm\ref{MFTP:A1}--\ref{MFTP:A2}} hold. Then
$$\lim_{T\to\i} {1\over T}V_{\scT}(x) = V, \q\forall x\in\dbR^n. $$
\end{corollary}

Let $(X^*_{\scT}(\cd),u^*_{\scT}(\cd))$ be the optimal pair of Problem (MFLQ)$_{\scT}$ for the initial state $x$,
and let $(Y^*_{\scT}(\cd),Z^*_{\scT}(\cd))$ be the {\it adjoint process}, i.e., the adapted solution
to the following mean-field backward SDE:
\begin{equation}\label{eqn:adjoint}\left\{\begin{aligned}
&dY^*_{\scT}(t) = -\,\big\{A^\top Y^*_{\scT}(t)+\bA^\top\dbE[Y^*_{\scT}(t)]+C^\top Z^*_{\scT}(t)+\bC^\top\dbE[Z^*_{\scT}(t)] + QX^*_{\scT}(t) \\
&\hp{dY^*_{\scT}(t)=} +\bQ\dbE[X^*_{\scT}(t)]+ S^\top u^*_{\scT}(t) + \bS^\top\dbE[u^*_{\scT}(t)]+q\big\} dt + Z^*_{\scT}(t)dW(t), \\
&Y^*_{\scT}(T) = 0.
\end{aligned}\right.\end{equation}
%

The next corollary shows that a strong exponential turnpike property also holds for the adjoint process.

\begin{corollary}\label{crlry:TP-adjoint}
Let {\rm\ref{MFTP:A1}--\ref{MFTP:A2}} hold. Let
\begin{align}\label{bmY+bmZ}
\bm Y^*(t) \deq PX^*(t) + \l^*, \q \bm Z^*(t) \deq P(C+D\Th) X^*(t) + P\si^*.
\end{align}
Then for some constants $K,\l>0$ independent of $T$,
\begin{align}\label{TP:Y*T-Z*T}
\dbE|Y^*_{\scT}(t)-\bm Y^*(t)|^2 + \dbE|Z^*_{\scT}(t)-\bm Z^*(t)|^2
\les K\[e^{-\l t} + e^{-\l(T-t)}\], \q\forall t\in[0,T].
\end{align}
\end{corollary}

\section{Stability of Riccati Equations}\label{Sec:stability}

As mentioned in \autoref{Sec:Pre}, there is a certain connection between the solutions to
the differential Riccati equation \rf{Ric:Pi} and the ARE \rf{ARE:Pi}.
We now examine this connection.
The main result of this section is as follows, which plays a central role in the study
of turnpike properties.

\begin{theorem}\label{thm:vP-Pi:bound}
Let {\rm\ref{MFTP:A1}--\ref{MFTP:A2}} hold.
Let $\vP_{\scT}(\cd)$ and $\vP$ be the unique solutions of \rf{Ric:Pi} and \rf{ARE:Pi}, respectively.
Then there exist positive constants $K,\l>0$ such that
\begin{equation}\label{Pi:shoulian}
|\vP_{\scT}(t)-\vP| \les Ke^{-\l(T-t)},  \q\forall 0\les t\les T<\i.
\end{equation}
\end{theorem}

Before going further, let us point out that in order to prove \autoref{thm:vP-Pi:bound},
we can assume without loss of generality that
\begin{align}\label{S=bS=0}
S=0, \q \bS=0.
\end{align}
Indeed, let
\begin{align}\label{Tihuan:1}
\cA = A-BR^{-1}S, \q \cC = C-DR^{-1}S, \q  \cQ = Q-S^\top R^{-1}S,
\end{align}
and let $\cP(\cd)$ be the positive semi-definite solution to the Riccati equation
\begin{equation}\label{Ric:cP}\left\{\begin{aligned}
&\dot\cP + \cP\cA + \cA^\top\cP + \cC^\top\cP\cC + \cQ  \\
&\hp{\cP} -(\cP B+ \cC^\top\cP D)(R+D^\top\cP D)^{-1}(B^\top\cP+D^\top\cP\cC)=0, \\
&\cP(T) = 0.
\end{aligned}\right.\end{equation}
We have
\begin{align}\label{eqn:cP-1}
& \cP\cA + \cA^\top\cP + \cC^\top\cP\cC + \cQ \nn\\
&\q= \cP A + A^\top\cP  + C^\top\cP C + Q -(\cP B+C^\top\cP D+S^\top)R^{-1}S  \nn\\
&\q\hp{=\ } - S^\top R^{-1}(B^\top\cP+D^\top\cP C+S) + S^\top R^{-1}(R+D^\top\cP D)R^{-1}S .
\end{align}
On the other hand,
\begin{align*}
& \cP B+ \cC^\top\cP D= (\cP B+C^\top\cP D+S^\top) -S^\top R^{-1}(R+D^\top\cP D),
\end{align*}
and hence
\begin{align}\label{eqn:cP-2}
& (\cP B+ \cC^\top\cP D)(R+D^\top\cP D)^{-1}(B^\top\cP+D^\top\cP\cC) \nn\\
&\q= (\cP B+C^\top\cP D+S^\top)(R+D^\top\cP D)^{-1}(B^\top\cP+D^\top\cP C+S) \nn\\
&\q\hp{=\ } -(\cP B+C^\top\cP D+S^\top)R^{-1}S - S^\top R^{-1}(B^\top\cP+D^\top\cP C+S) \nn\\
&\q\hp{=\ } +S^\top R^{-1}(R+D^\top\cP D)R^{-1}S.
\end{align}
Substituting \rf{eqn:cP-1} and \rf{eqn:cP-2} into \rf{Ric:cP},
we see that $\cP(\cd)$ satisfies the same equation as $P_{\scT}(\cd)$, which implies $\cP(\cd)=P_{\scT}(\cd)$.
By a similar argument, we can show that $\vP_{\scT}(\cd)$ satisfies
\begin{equation}\label{Ric:cPi}\left\{\begin{aligned}
&\dot\vP_{\scT} + \vP_{\scT}\wh\cA + \wh\cA\,^\top\vP_{\scT} + \wh\cC\,^\top P_{\scT}\wh\cC + \wh\cQ  \\
&\hp{\vP_{\scT}} -(\vP_{\scT}\hB+ \wh\cC\,^\top P_{\scT}\hD)(\hR+\hD^\top P_{\scT}\hD)^{-1}(\hB^\top\vP_{\scT}+\hD^\top P_{\scT}\wh\cC)=0, \\
&\vP_{\scT}(T) = 0,
\end{aligned}\right.\end{equation}
where
\begin{align}\label{Tihuan:2}
\wh{\cA} = \hA - \hB\hR^{-1}\hS, \q
\wh{\cC} = \hC - \hD\hR^{-1}\hS, \q
\wh{\cQ} = \hQ - \hS^\top\hR^{-1}\hS.
\end{align}
Thus, by making transformations \rf{Tihuan:1} and \rf{Tihuan:2}, we can assume without loss of generality
that \rf{S=bS=0} holds.

\ms

To prove \autoref{thm:vP-Pi:bound}, we need the following lemma.

\begin{lemma}\label{lmm:DLQ-matrix}
Let {\rm\ref{MFTP:A1}} hold. Then for any positive semi-definite matrix $\D\ges0$,
$$\begin{pmatrix*}\hC^\top\D\hC+\hQ & \hC^\top\D\hD \\
                  \hD^\top\D\hC     & \wh\cR(\D)\end{pmatrix*}\ges0, $$
or equivalently,
$$ \hC^\top\D\hC+\hQ- \(\hC^\top\D\hD\) \wh\cR(\D)^{-1} \(\hD^\top\D\hC\) \ges0. $$
\end{lemma}

\begin{proof}
Since $\hR>0$ and $\D\ges0$, we have
\begin{align}\label{Up:positive-1}
\wh\cR(\D) = \hR+\hD^\top\D\hD \ges \hR >0.
\end{align}
Let $M\in\dbR^{n\times n}$ be such that $\D=M^\top M$. Then by \autoref{lmm:matrix},
$$ I- \(M\hD\) \wh\cR(\D)^{-1} \(\hD^\top M^\top\)
=  I - \(M\hD\)\[\hR+\(M\hD\)^\top\(M\hD\)\]^{-1}\(M\hD\)^\top \ges 0. $$
Consequently,
\begin{align*}
& \hC^\top\D\hC+\hQ- \big(\hC^\top\D\hD\big) \wh\cR(\D)^{-1} \big(\hD^\top\D\hC\big) \\
&\q= \hQ + \hC^\top\[\D- \big(\D\hD\big) \wh\cR(\D)^{-1} \big(\hD^\top\D\big)\]\hC  \\
&\q= \hQ + \big(M\hC\big)^\top\[I- \big(M\hD\big) \wh\cR(\D)^{-1} \big(\hD^\top M^\top\big)\]\big(M\hC\big)\ges 0.
\end{align*}
The proof is complete.
\end{proof}

In the following proof and in the sequel, we shall denote by $K$ and $\l$ two generic positive constants,
which do not depend on $T$ and may vary from line to line.

\begin{proof}[\textbf{Proof of \autoref{thm:vP-Pi:bound}.}]
As mentioned previously, we may assume $S=0$ and $\bS=0$ to simplify the proof.
Define for $t\in[0,\i)$,
\begin{align*}
\Si(t) \deq P_{\scT}(T-t), \q \Up(t) \deq \vP_{\scT}(T-t); \q\text{if}~t\les T.
\end{align*}
It is shown in \cite{Sun-Wang-Yong2022} that the definition of $\Si(t)$ does not depend on the choice of $T\ges t$, and
\begin{align}\label{Si:dandiao}
 0<\Si(s)\les \Si(t), \q |\Si(t)-P| \les Ke^{-\l t}, \q\forall 0\les s\les t<\i,
\end{align}
for some positive constants $K,\l>0$.
A similar argument shows that the definition of $\Up(t)$ does not depend on the choice of $T\ges t$ either,
and $\Up(\cd)$ satisfies the following ODE:
\begin{equation}\label{Ric:Up}\left\{\begin{aligned}
& \dot\Up(t) = \wh\cQ(\Si(t),\Up(t))
  -\wh\cS(\Si(t),\Up(t))^\top\wh\cR(\Si(t))^{-1}\wh\cS(\Si(t),\Up(t)), \\
& \Up(0) = 0.
\end{aligned}\right.\end{equation}

\ss

{\it Step 1.} We first show that $\lim_{t\to\i}\Up(t)=\vP$.

\ms

Consider the deterministic LQ optimal control problem with state equation
\begin{equation*}\left\{\begin{aligned}
& \dot X(t) = \hA X(t) +\hB u(t), \q t\in[0,T],\\
& X(0) = x,
\end{aligned}\right.\end{equation*}
and cost functional
\begin{align*}
\cJ_{\scT}(x;u(\cd)) &= \int_0^T\Blan
\begin{pmatrix*} \hC^\top P_{\scT}(t)\hC+\hQ & \hC^\top P_{\scT}(t)\hD \\
                 \hD^\top P_{\scT}(t)\hC     & \wh\cR(P_{\scT}(t))\end{pmatrix*}
\begin{pmatrix*} X(t) \\ u(t)\end{pmatrix*},
\begin{pmatrix*} X(t) \\ u(t)\end{pmatrix*}\Bran dt.
\end{align*}
Since $P_{\scT}(t)>0$,
\begin{align}\label{Up:positive-1}
\wh\cR(P_{\scT}(t)) = \hR+\hD^\top P_{\scT}(t)\hD \ges \hR >0,
\end{align}
and by \autoref{lmm:DLQ-matrix},
\begin{align}\label{Up:positive-2}
\begin{pmatrix*} \hC^\top P_{\scT}(t)\hC+\hQ & \hC^\top P_{\scT}(t)\hD \\
                 \hD^\top P_{\scT}(t)\hC     & \wh\cR(P_{\scT}(t))\end{pmatrix*}\ges0.
\end{align}
Because of \rf{Up:positive-1} and \rf{Up:positive-2}, the above deterministic LQ problem
is uniquely solvable with value function given by $\lan\vP_{\scT}(0)x,x\ran$.
Note that $\cJ_{\scT}(x;u(\cd))$ can be written as
\begin{align*}
\cJ_{\scT}(x;u(\cd)) &= \int_0^T\Blan
\begin{pmatrix*} \hC^\top \Si(T-t)\hC+\hQ & \hC^\top \Si(T-t)\hD \\
                 \hD^\top \Si(T-t)\hC     & \wh\cR(\Si(T-t))\end{pmatrix*}
\begin{pmatrix*} X(t) \\ u(t)\end{pmatrix*},
\begin{pmatrix*} X(t) \\ u(t)\end{pmatrix*}\Bran dt.
\end{align*}
For $T_2\ges T_1\ges0$, set
$$ \D(t) = \Si(T_2-t)-\Si(T_1-t), \q 0\les t\les T_1. $$
Since $\Si(\cd)$ is nondecreasing, $\D(t)\ges0$ and hence by \autoref{lmm:DLQ-matrix},
$$\begin{pmatrix*}
\hC^\top \D(t)\hC+\hQ & \hC^\top \D(t)\hD \\
\hD^\top \D(t)\hC     & \wh\cR(\D(t))
\end{pmatrix*} \ges 0. $$
Thus, for any square-integrable $u:[0,T_2]\to\dbR^m$,
\begin{align*}
&\cJ_{\scT_2}(x;u(\cd))-\cJ_{\scT_1}(x;u(\cd)) \\
&\q\ges \int_0^{\scT_1}\Blan
\begin{pmatrix*} \hC^\top \Si(T_2-t)\hC+\hQ & \hC^\top \Si(T_2-t)\hD \\
                 \hD^\top \Si(T_2-t)\hC     & \wh\cR(\Si(T_2-t))\end{pmatrix*}
\begin{pmatrix*} X(t) \\ u(t)\end{pmatrix*},
\begin{pmatrix*} X(t) \\ u(t)\end{pmatrix*}\Bran dt \\
&\q\hp{=\ } -\int_0^{\scT_1}\Blan
\begin{pmatrix*} \hC^\top \Si(T_1-t)\hC+\hQ & \hC^\top \Si(T_1-t)\hD \\
                 \hD^\top \Si(T_1-t)\hC     & \wh\cR(\Si(T_1-t))\end{pmatrix*}
\begin{pmatrix*} X(t) \\ u(t)\end{pmatrix*},
\begin{pmatrix*} X(t) \\ u(t)\end{pmatrix*}\Bran dt \\
&\q= \int_0^{\scT_1}\Blan
\begin{pmatrix*} \hC^\top \D(t)\hC+\hQ & \hC^\top \D(t)\hD \\
                 \hD^\top \D(t)\hC     & \wh\cR(\D(t))\end{pmatrix*}
\begin{pmatrix*} X(t) \\ u(t)\end{pmatrix*},
\begin{pmatrix*} X(t) \\ u(t)\end{pmatrix*}\Bran dt \\
&\q\ges0.
\end{align*}
This implies that for $0\les T_1\les T_2$,
$$ \lan\vP_{\scT_1}(0)x,x\ran \les \lan\vP_{\scT_2}(0)x,x\ran, \q\forall x\in\dbR^n. $$
Since $x$ is arbitrary, $\Up(t)=\vP_t(0)$ is nondecreasing in $t$.
On the other hand, take $\wh\Th$ to be a stabilizer of the system \rf{ODE-sys}.
By \rf{Si:dandiao}, $P_{\scT}(t)\les P$ and hence (recalling \autoref{lmm:DLQ-matrix})
\begin{align*}
\lan\vP_{\scT}(0)x,x\ran
&\les \int_0^T\Blan
\begin{pmatrix*} \hC^\top P_{\scT}(t)\hC+\hQ & \hC^\top P_{\scT}(t)\hD \\
                 \hD^\top P_{\scT}(t)\hC     & \wh\cR(P_{\scT}(t))\end{pmatrix*}
\begin{pmatrix*} X(t) \\ \wh\Th X(t)\end{pmatrix*},
\begin{pmatrix*} X(t) \\ \wh\Th X(t)\end{pmatrix*}\Bran dt\\
&\les \int_0^T\Blan
\begin{pmatrix*} \hC^\top P\hC+\hQ & \hC^\top P\hD \\
                 \hD^\top P\hC     & \wh\cR(P)\end{pmatrix*}
\begin{pmatrix*} X(t) \\ \wh\Th X(t)\end{pmatrix*},
\begin{pmatrix*} X(t) \\ \wh\Th X(t)\end{pmatrix*}\Bran dt\\
&\les \int_0^\i\Blan
\begin{pmatrix*} \hC^\top P\hC+\hQ & \hC^\top P\hD \\
                 \hD^\top P\hC     & \wh\cR(P)\end{pmatrix*}
\begin{pmatrix*} X(t) \\ \wh\Th X(t)\end{pmatrix*},
\begin{pmatrix*} X(t) \\ \wh\Th X(t)\end{pmatrix*}\Bran dt<\i.
\end{align*}
Thus, $\Up(t)=\vP_t(0)$ is bounded above, which, together with the monotonicity of $\Up(t)$, implies that
$$ \Up_\i \deq \lim_{t\to\i}\Up(t) $$
exists and is finite. From \rf{Ric:Up} we have
\begin{align*}
\Up(t+1)-\Up(t)
&= \int_t^{t+1} \wh\cQ(\Si(s),\Up(s))
  -\wh\cS(\Si(s),\Up(s))^\top\wh\cR(\Si(s))^{-1}\wh\cS(\Si(s),\Up(s)) ds.
\end{align*}
Letting $t\to\i$ yields
$$ \wh\cQ(P,\Up_\i)-\wh\cS(P,\Up_\i)^\top\wh\cR(P)^{-1}\wh\cS(P,\Up_\i)=0. $$
By uniqueness, $\Up_\i=\vP$.

\ms

{\it Step 2.} Set $\D(t)\deq \Up(t)-\vP$ and $\L(t)\deq \Si(t)-P$.
We show that $\D(\cd)$ satisfies the following ODE:
\begin{equation}\label{dotD=*}
\dot\D(t) = \D(t)\bar\cA + \bar\cA^\top\D(t) + f(\D(t),\L(t)) + g(\L(t)),
\end{equation}
where
\begin{align*}
\bar\cA &\deq \hA+\hB\wh\Th \q\text{with}\q \wh\Th \deq -\wh\cR(P)^{-1}\wh\cS(P,\vP), \\[2mm]
f(\D(t),\L(t)) &\deq -\wh\cS(\L(t),\D(t))^\top\wh\cR(\Si(t))^{-1}\wh\cS(\L(t),\D(t)) \\
&\hp{=\ } -\[(\hD\wh\Th)^\top\L(t)\hD\wh\cR(\Si(t))^{-1}\hB^\top\D(t)\]^\top \\
&\hp{=\ } -\[(\hD\wh\Th)^\top\L(t)\hD\wh\cR(\Si(t))^{-1}\hB^\top\D(t)\], \\[2mm]
g(\L(t)) &\deq \hC^\top\L(t)\hC - (\hD\wh\Th)^\top\L(t)\hD\wh\cR(\Si(t))^{-1}\wh\cS(P,\vP) \\
&\hp{=\ } -\[\wh\cS(P,\vP)^\top\wh\cR(\Si(t))^{-1}\hD^\top\L(t)\hC\]^\top \\
&\hp{=\ } -\[\wh\cS(P,\vP)^\top\wh\cR(\Si(t))^{-1}\hD^\top\L(t)\hC\].
\end{align*}

\ss

To verify \rf{dotD=*}, we first observe that
\begin{equation}\label{dotD=}\begin{aligned}
\dot\D(t)
&= \D(t)\hA + \hA^\top\D(t) + \hC^\top\L(t)\hC  \\
&\hp{=\ } -\[\hB^\top\Up(t)+\hD^\top\Si(t)\hC\]^\top\[\hR+\hD^\top\Si(t)\hD\]^{-1}\[\hB^\top\Up(t)+\hD^\top\Si(t)\hC\], \\
&\hp{=\ } +\[\hB^\top\vP+\hD^\top P\hC\]^\top\[\hR+\hD^\top P\hD\]^{-1}\[\hB^\top\vP+\hD^\top P\hC\].
\end{aligned}\end{equation}
Noting that
\begin{align*}
& \hB^\top\Up+\hD^\top\Si(t)\hC = \[\hB^\top\D(t)+\hD^\top\L(t)\hC\] + \[\hB^\top\vP+\hD^\top P\hC\],
\end{align*}
we have
\begin{align*}
& \[\hB^\top\Up(t)+\hD^\top\Si(t)\hC\]^\top\[\hR+\hD^\top\Si(t)\hD\]^{-1}\[\hB^\top\Up(t)+\hD^\top\Si(t)\hC\] \\
&\q= \[\hB^\top\D(t)+\hD^\top\L(t)\hC\]^\top\[\hR+\hD^\top\Si(t)\hD\]^{-1}\[\hB^\top\D(t)+\hD^\top\L(t)\hC\] \\
&\q\hp{=\ } + \[\hB^\top\vP+\hD^\top P\hC\]^\top\[\hR+\hD^\top\Si(t)\hD\]^{-1}\[\hB^\top\vP+\hD^\top P\hC\] \\
&\q\hp{=\ } + \[\hB^\top\D(t)+\hD^\top\L(t)\hC\]^\top\[\hR+\hD^\top\Si(t)\hD\]^{-1}\[\hB^\top\vP+\hD^\top P\hC\] \\
&\q\hp{=\ } + \[\hB^\top\vP+\hD^\top P\hC\]^\top\[\hR+\hD^\top\Si(t)\hD\]^{-1}\[\hB^\top\D(t)+\hD^\top\L(t)\hC\] \\
&\q\equiv \circled{1} + \circled{2} + \circled{3} + \circled{4}.
\end{align*}
For $\circled{2}$ and $\circled{4}$, we note that
\begin{align*}
& \[\hR+\hD^\top\Si(t)\hD\]^{-1} - \[\hR+\hD^\top P\hD\]^{-1} \\
&\q= \[\hR+\hD^\top P\hD\]^{-1}\[\(\hR+\hD^\top P\hD\)-\(\hR+\hD^\top\Si(t)\hD\)\]\[\hR+\hD^\top\Si(t)\hD\]^{-1}   \\
&\q= -\[\hR+\hD^\top P\hD\]^{-1}\hD^\top\L(t)\hD \[\hR+\hD^\top\Si(t)\hD\]^{-1}.
\end{align*}
Thus,
\begin{align*}
\circled{2}
&= \[\hB^\top\vP+\hD^\top P\hC\]^\top\[\hR+\hD^\top P\hD\]^{-1}\[\hB^\top\vP+\hD^\top P\hC\]  \\
&\hp{=\ } -\[\hB^\top\!\vP\!+\!\hD^\top\!P\hC\]^\top\[\hR\!+\!\hD^\top P\hD\]^{-1}\hD^\top\L(t)\hD
           \[\hR\!+\!\hD^\top\Si(t)\hD\]^{-1}\[\hB^\top\!\vP\!+\!\hD^\top\!P\hC\] \\
&= \[\hB^\top\vP+\hD^\top P\hC\]^\top\[\hR+\hD^\top P\hD\]^{-1}\[\hB^\top\vP+\hD^\top P\hC\]  \\
&\hp{=\ } +(\hD\wh\Th)^\top\L(t)\hD\wh\cR(\Si(t))^{-1}\wh\cS(P,\vP),
\end{align*}
and
\begin{align*}
\circled{4}
&= \[\hB^\top\vP+\hD^\top P\hC\]^\top\[\hR+\hD^\top\Si(t)\hD\]^{-1}\hD^\top\L(t)\hC \\
&\hp{=\ } +\[\hB^\top\vP+\hD^\top P\hC\]^\top\[\hR+\hD^\top\Si(t)\hD\]^{-1}\hB^\top\D(t) \\
&= \[\hB^\top\vP+\hD^\top P\hC\]^\top\[\hR+\hD^\top\Si(t)\hD\]^{-1}\hD^\top\L(t)\hC \\
&\hp{=\ } +\[\hB^\top\vP+\hD^\top P\hC\]^\top\[\hR+\hD^\top P\hD\]^{-1}\hB^\top\D(t) \\
&\hp{=\ } -\[\hB^\top\vP+\hD^\top P\hC\]^\top\[\hR+\hD^\top P\hD\]^{-1}\hD^\top\L(t)\hD \[\hR+\hD^\top\Si(t)\hD\]^{-1}\hB^\top\D(t) \\
&= \wh\cS(P,\vP)^\top\wh\cR(\Si(t))^{-1}\hD^\top\L(t)\hC - (\hB\wh\Th)^\top\D(t)
   +(\hD\wh\Th)^\top\L(t)\hD\wh\cR(\Si(t))^{-1}\hB^\top\D(t).
\end{align*}
Substituting these into \rf{dotD=} and noting that $\circled{3}=\circled{4}^\top$, we obtain
\begin{align*}
\dot\D(t)
&= \D(t)\hA + \hA^\top\D(t) + \hC^\top\L(t)\hC  \\
&\hp{=\ } -\[\hB^\top\D(t)+\hD^\top\L(t)\hC\]^\top\[\hR+\hD^\top\Si(t)\hD\]^{-1}\[\hB^\top\D(t)+\hD^\top\L(t)\hC\] \\
&\hp{=\ } -(\hD\wh\Th)^\top\L(t)\hD\wh\cR(\Si(t))^{-1}\wh\cS(P,\vP) \\
&\hp{=\ } -\[\wh\cS(P,\vP)^\top\wh\cR(\Si(t))^{-1}\hD^\top\!\L(t)\hC \!-\! (\hB\wh\Th)^\top\!\D(t)
          \!+\! (\hD\wh\Th)^\top\!\L(t)\hD\wh\cR(\Si(t))^{-1}\hB^\top\!\D(t)\]^\top  \\
&\hp{=\ } -\[\wh\cS(P,\vP)^\top\wh\cR(\Si(t))^{-1}\hD^\top\!\L(t)\hC \!-\! (\hB\wh\Th)^\top\!\D(t)
          \!+\! (\hD\wh\Th)^\top\!\L(t)\hD\wh\cR(\Si(t))^{-1}\hB^\top\!\D(t)\] \\
&= \D(t)\(\hA+\hB\wh\Th\) + \(\hA+\hB\wh\Th\)^\top\D(t) + \hC^\top\L(t)\hC  \\
&\hp{=\ } -\[\hB^\top\D(t)+\hD^\top\L(t)\hC\]^\top\[\hR+\hD^\top\Si(t)\hD\]^{-1}\[\hB^\top\D(t)+\hD^\top\L(t)\hC\] \\
&\hp{=\ } -\[(\hD\wh\Th)^\top\L(t)\hD\wh\cR(\Si(t))^{-1}\hB^\top\D(t)\]^\top
          -\[(\hD\wh\Th)^\top\L(t)\hD\wh\cR(\Si(t))^{-1}\hB^\top\D(t)\] \\
&\hp{=\ } -(\hD\wh\Th)^\top\L(t)\hD\wh\cR(\Si(t))^{-1}\wh\cS(P,\vP) \\
&\hp{=\ } -\[\wh\cS(P,\vP)^\top\wh\cR(\Si(t))^{-1}\hD^\top\L(t)\hC\]^\top
          -\[\wh\cS(P,\vP)^\top\wh\cR(\Si(t))^{-1}\hD^\top\L(t)\hC\],
\end{align*}
which is exactly \rf{dotD=*}.

\ms

{\it Step 3.} Next we show that there exist positive constants $K,\l>0$ such that
\begin{equation}\label{Up:shoulian}
|\Up(t)-\vP| \les Ke^{-\l t}, \q\forall 0\les t\les T<\i,
\end{equation}
which is equivalent to \rf{Pi:shoulian}.

\ms

By the variation of constants formula, for $t\ges s\ges0$,
\begin{align}\label{D-formula}
%
\D(t) &= e^{\bar\cA^\top(t-s)}\lt\{\D(s) \!+\! \int_s^t e^{\bar\cA^\top(s-r)}\[f(\D(r),\L(r)) \!+\! g(\L(r))\]e^{\bar\cA(s-r)} dr\rt\}e^{\bar\cA(t-s)}.
\end{align}
Clearly, there exists a constant $\rho>0$ such that (noting that $|\L(t)|$ is bounded so that $|\L(t)|^2\les K|\L(t)|$)
\begin{align}\label{fg:bound}
|f(\D(t),\L(t))| + |g(\L(t))| &\les \rho\[|\D(t)|^2+|\L(t)|\].
\end{align}
By \autoref{lmm:ARE-sol}, $\bar\cA$ is stable, so
$$ |e^{\bar\cA t}| \les \a_1e^{-\b_1 t}, \q\forall t\ges0 $$
for some constants $\a_1,\b_1>0$. Also, by \autoref{lmm:Si-P:bound},
$$ |\L(t)| = |\Si(t)-P| \les \a_2e^{-3\b_2 t}, \q\forall t\ges0 $$
for some positive constants $\a_2,\b_2>0$.
Take $\a=\a_1\vee \a_2$ and $\b=\b_1\wedge\b_2$ so that
\begin{align}\label{cA+Lambda:bound}
|e^{\bar\cA t}| \les \a e^{-\b t}, \q |\L(t)| \les \a e^{-3\b t}, \q\forall t\ges0.
\end{align}
From \rf{D-formula}, \rf{fg:bound}, and \rf{cA+Lambda:bound} we get
\begin{align*}
|\D(t)| &\les \a^2e^{-2\b(t-s)}\lt\{|\D(s)| + \rho\int_s^t e^{2\b(r-s)}\[|\D(r)|^2+|\L(r)|\] dr\rt\} \\
&\les \a^2e^{-2\b(t-s)}\lt\{|\D(s)| + \rho\int_s^t e^{2\b(r-s)}\[|\D(r)|^2+\a e^{-3\b r}\] dr\rt\} \\
&= \a^2e^{-2\b(t-s)}\lt\{|\D(s)| + \rho\a e^{-2\b s}\b^{-1}\[e^{-\b s}-e^{-\b t}\]+ \rho\int_s^t e^{2\b(r-s)}|\D(r)|^2 dr\rt\} \\
&\les \a^2e^{-2\b(t-s)}\lt\{|\D(s)| + \rho\a e^{-2\b s}\b^{-1}+ \rho\int_s^t e^{2\b(r-s)}|\D(r)|^2 dr\rt\}.
\end{align*}
For fixed $s\ges0$, set
$$ h(t) = e^{2\b(t-s)}|\D(t)|, \q k(s) = \a^2|\D(s)| + \rho\a^3e^{-2\b s}\b^{-1}. $$
Then
\begin{align*}
h(t) &\les k(s) + \rho\a^2\int_s^t e^{-2\b(r-s)}[h(r)]^2 dr, \q t\ges s.
\end{align*}
Set
$$ \th(t) = k(s) + \rho\a^2\int_s^t e^{-2\b(r-s)}[h(r)]^2 dr, \q p(t) = {1\over \th(t)}; \q t\ges s. $$
Then
\begin{align*}
p'(t) &= -{\th'(t)\over [\th(t)]^2} = -{\rho\a^2 e^{-2\b(t-s)}[h(t)]^2\over [\th(t)]^2} \ges -\rho\a^2 e^{-2\b(t-s)}.
\end{align*}
Since $\ds\lim_{s\to\i}k(s)=0$,  for $s$ large enough such that
$$ p(s) = {1 \over \th(s)} = {1 \over k(s)} \ges 1 + {\rho\a^2 \over 2\b}, $$
we have
\begin{align*}
p(t) &\ges p(s)-\rho\a^2\int_s^t e^{-2\b(r-s)}dr = p(s)-{\rho\a^2 \over 2\b}\[1-e^{2\b(s-t)}\] \ges 1
\end{align*}
and hence
$$ h(t) \les \th(t) = {1\over p(t)} \les 1, \q\forall t\ges s. $$
It follows that
$$ |\D(t)| = e^{-2\b(t-s)}h(t) \les K e^{-\l t}, \q\forall t\ges0, $$
for some constants $K,\l>0$.
\end{proof}

\section{The Turnpike Property}\label{Sec:proof}

In this section we prove the main results stated in \autoref{Sec:main-result}.
According to \cite[Theorem 2.3]{Sun2017}, under the assumption \ref{MFTP:A1}, Problem (MFLQ)$_{\scT}$
admits a unique optimal control for every initial state $x$.
Moreover, a pair $(X^*_{\scT}(\cd),u^*_{\scT}(\cd))$ is optimal for $x$ if and only if it satisfies
the following optimality system:
\begin{equation}\label{os:ProbT}\left\{\begin{aligned}
&dX^*_{\scT}(t) = \big\{AX^*_{\scT}(t)+\bA\dbE[X^*_{\scT}(t)]+Bu^*_{\scT}(t)+\bB\dbE[u^*_{\scT}(t)]+b\big\}dt \\
&\hp{dX^*_{\scT}(t)=} +\big\{CX^*_{\scT}(t)+\bC\dbE[X^*_{\scT}(t)]+Du^*_{\scT}(t)+\bD\dbE[u^*_{\scT}(t)]+\si\big\}dW(t),  \\
&dY^*_{\scT}(t) = -\,\big\{A^\top Y^*_{\scT}(t)+\bA^\top\dbE[Y^*_{\scT}(t)]+C^\top Z^*_{\scT}(t)+\bC^\top\dbE[Z^*_{\scT}(t)] + QX^*_{\scT}(t) \\
&\hp{dY^*_{\scT}(t)=} +\bQ\dbE[X^*_{\scT}(t)]+ S^\top u^*_{\scT}(t) + \bS^\top\dbE[u^*_{\scT}(t)]+q\big\} dt + Z^*_{\scT}(t)dW(t), \\
&X^*_{\scT}(0) = x, \q Y^*_{\scT}(T) = 0, \\
&B^\top Y^*_{\scT}(t)+\bB^\top\dbE[Y^*_{\scT}(t)]+D^\top Z^*_{\scT}(t)+\bD^\top\dbE[Z^*_{\scT}(t)] + SX^*_{\scT}(t) + \bS\dbE[X^*_{\scT}(t)] \\
&\hp{B^\top Y^*_{\scT}(t)} + Ru^*_{\scT}(t) + \bR\dbE[u^*_{\scT}(t)] + r = 0.
\end{aligned}\right.\end{equation}

\ms

Let $(X^*_{\scT}(\cd),u^*_{\scT}(\cd))$ be the optimal pair of Problem (MFLQ)$_{\scT}$
for the initial state $x$ and $(Y^*_{\scT}(\cd),Z^*_{\scT}(\cd))$ the adapted solution
to the corresponding adjoint equation in \rf{os:ProbT}.
Let $(x^*,u^*)$ be the unique solution of Problem (O) and $\l^*\in\dbR^n$ the vector in \rf{l+mu:exist}.
Define
\begin{align}\label{def:hX+hu}
\wX_{\scT}(\cd) = X^*_{\scT}(\cd)-x^*,  \q
\tu_{\scT}(\cd) = u^*_{\scT}(\cd)-u^*,  \q
\wY_{\scT}(\cd) = Y^*_{\scT}(\cd)-\l^*.
\end{align}
Subtracting \rf{l+mu:exist} from \rf{os:ProbT} yields
\begin{equation}\label{os:wX}\left\{\begin{aligned}
&d\wX_{\scT}(t) = \big\{A\wX_{\scT}(t)+\bA\dbE[\wX_{\scT}(t)]+B\tu_{\scT}(t)+\bB\dbE[\tu_{\scT}(t)]\big\}dt \\
&\hp{d\wX_{\scT}(t)=} +\big\{C\wX_{\scT}(t)+\bC\dbE[\wX_{\scT}(t)]+D\tu_{\scT}(t)+\bD\dbE[\tu_{\scT}(t)]+\si^*\big\}dW(t),  \\
&d\wY_{\scT}(t) = -\,\big\{A^\top\wY_{\scT}(t)+\bA^\top\dbE[\wY_{\scT}(t)]+C^\top Z^*_{\scT}(t)+\bC^\top\dbE[Z^*_{\scT}(t)]+Q\wX_{\scT}(t) \\
&\hp{d\wY_{\scT}(t)=} +\bQ\dbE[\wX_{\scT}(t)]+ S^\top\tu_{\scT}(t) + \bS^\top\dbE[\tu_{\scT}(t)]-\hC^\top P\si^*\big\} dt + Z^*_{\scT}(t)dW(t), \\
&\wX_{\scT}(0) = x-x^*, \q \wY_{\scT}(T) = -\l^*, \\
& B^\top\wY_{\scT}(t) + \bB^\top\dbE[\wY_{\scT}(t)] + D^\top Z^*_{\scT}(t) + \bD^\top\dbE[Z^*_{\scT}(t)]
                      + S\wX_{\scT}(t) + \bS\dbE[\wX_{\scT}(t)] \\
&\hp{B^\top\wY_{\scT}(t)} + R\tu_{\scT}(t) + \bR\dbE[\tu_{\scT}(t)] - \hD^\top P\si^* = 0,
\end{aligned}\right.\end{equation}
where $\si^* = \hC x^*+ \hD u^* + \si$.
From \rf{os:wX} we see that $(\wX_{\scT}(\cd),\tu_{\scT}(\cd))$ is an optimal pair of the mean-field stochastic LQ problem
with state equation
\begin{equation}\label{New-state}\left\{\begin{aligned}
dX(t) &= \big\{AX(t)+\bA\dbE[X(t)]+Bu(t)+\bB\dbE[u(t)]\big\}dt \\
      &\hp{=\ } +\big\{CX(t)+\bC\dbE[X(t)]+Du(t)+\bD\dbE[u(t)]+\si^*\big\}dW(t), \\
 X(0) &= x-x^*,
\end{aligned}\right.\end{equation}
and cost functional
\begin{align*}
J(x;u) &= \dbE\bigg\{\!-2\lan\l^*,X(T)\ran + \!\int_0^T\!\bigg[\Blan
\begin{pmatrix*}[l]Q & \!\!S^\top \\ S & \!\!R\end{pmatrix*}\!\!
\begin{pmatrix}X(t) \\ u(t)\end{pmatrix}\!,\!
\begin{pmatrix}X(t) \\ u(t)\end{pmatrix}\Bran
-2\Blan\begin{pmatrix}C^\top\! P\si^* \\ D^\top\! P\si^*\end{pmatrix}\!,\!
       \begin{pmatrix}X(t) \\ u(t)\end{pmatrix}\Bran \\
&\hp{=\ } +\Blan\begin{pmatrix*}[l]\bQ & \!\bS^\top \\ \bS & \!\bR\end{pmatrix*}\!
                \begin{pmatrix}\dbE[X(t)] \\ \dbE[u(t)]\end{pmatrix}\!,
                \begin{pmatrix}\dbE[X(t)] \\ \dbE[u(t)]\end{pmatrix}\Bran
-2\Blan\begin{pmatrix}\bC^\top P\si^* \\ \bD^\top P\si^*\end{pmatrix}\!,
       \begin{pmatrix}\dbE[X(t)] \\ \dbE[u(t)]\end{pmatrix}\Bran\bigg] dt\bigg\}.
\end{align*}
Thus, by \cite[Theorem 5.2]{Sun2017} we have the following result.

\begin{proposition}
Let {\rm\ref{MFTP:A1}--\ref{MFTP:A2}} hold. Let $P_{\scT}(\cd)$ be the solution to \rf{Ric:P}
and $\vP_{\scT}(\cd)$ the solution to \rf{Ric:Pi}. Define
\begin{align}
\label{def:ThT}
\Th_{\scT}(t)    &\deq  -\cR(P_{\scT}(t))^{-1}\cS(P_{\scT}(t)), \\
\label{def:wThT}
\wh\Th_{\scT}(t) &\deq -\wh\cR(P_{\scT}(t))^{-1}\wh\cS(P_{\scT}(t),\vP_{\scT}(t)),
\end{align}
and let $\f_{\scT}(\cd)$ and  $\wh\f_{\scT}(\cd)$ be the solutions to
\begin{equation}\label{BODE:f}\left\{\begin{aligned}
& \dot\f_{\scT}(t) + [A+B\Th_{\scT}(t)]^\top\f_{\scT}(t) + [C+D\Th_{\scT}(t)]^\top[P_{\scT}(t)-P]\si^*, \\
& \f_{\scT}(T) = -\l^*,
\end{aligned}\right.\end{equation}
and
\begin{equation}\label{BODE:bf}\left\{\begin{aligned}
& \dot{\wh\f}_{\scT}(t) + [\hA+\hB\wh\Th_{\scT}(t)]^\top\wh\f_{\scT}(t) + [\hC+\hD\wh\Th_{\scT}(t)]^\top[P_{\scT}(t)-P]\si^*, \\
& \wh\f_{\scT}(T) = -\l^*,
\end{aligned}\right.\end{equation}
respectively. Then the process $\tu_{\scT}(\cd)$ defined in \rf{def:hX+hu} is given by
\begin{align}\label{rep:hu}
 \tu_{\scT}(t) = \Th_{\scT}(t)\{\wX_{\scT}(t)-\dbE[\wX_{\scT}(t)]\}
                 +\wh\Th_{\scT}(t)\dbE[\wX_{\scT}(t)] + \th_{\scT}(t) + \wh\th_{\scT}(t),
\end{align}
where
\begin{align}
\th_{\scT}(t)    &= -\cR(P_{\scT}(t))^{-1}\big[B^\top\f_{\scT}(t) + D^\top(P_{\scT}(t)-P)\si^*\big],           \label{rep:f}\\
\wh\th_{\scT}(t) &= -\wh\cR(P_{\scT}(t))^{-1}\big[\hB^\top\wh\f_{\scT}(t) + \hD^\top(P_{\scT}(t)-P)\si^*\big]. \label{rep:wf}
\end{align}
\end{proposition}

\begin{lemma}\label{lmm:th:bound}
Let {\rm\ref{MFTP:A1}--\ref{MFTP:A2}} hold. Then there exist positive constants $K,\l>0$, independent of $T$, such that
$$ |\wh\f_{\scT}(t)| + |\th_{\scT}(t)| + |\wh\th_{\scT}(t)| \les Ke^{-\l(T-t)}, \q\forall t\in[0,T]. $$
\end{lemma}

\begin{proof}
It is shown in \cite{Sun-Wang-Yong2022} that $\f_{\scT}(\cd)$, the solution of \rf{BODE:f}, satisfies
$$ |\f_{\scT}(t)| \les Ke^{-\l(T-t)}, \q\forall t\in[0,T] $$
for some positive constants $K,\l>0$ independent of $T$.
In a similar manner, we can show that the solution $\wh\f_{\scT}(\cd)$ of \rf{BODE:bf} satisfies
the same inequality with possibly different constants $K$ and $\l$.
The desired result then follows from  the fact
$$\cR(P_{\scT}(t))\ges R>0, \q \wh\cR(P_{\scT}(t))\ges \hR>0 $$
and \autoref{lmm:Si-P:bound}.
\end{proof}

If we substitute \rf{rep:hu} into \rf{New-state}, we see that the process $\wX_{\scT}(\cd)$ defined
in \rf{def:hX+hu} satisfies the closed-loop system
\begin{equation}\label{state:cloop}\left\{\begin{aligned}
d\wX_{\scT}(t)
&= \big\{(A\!+\!B\Th_{\scT})(\wX_{\scT}\!-\!\dbE[\wX_{\scT}])\!+\!(\hA\!+\!\hB\wh\Th_{\scT})\dbE[\wX_{\scT}]
           \!+\!\hB(\th_{\scT}\!+\!\wh\th_{\scT})\big\}dt \\
&\hp{=\ } +\big\{(C\!+\!D\Th_{\scT})(\wX_{\scT}\!-\!\dbE[\wX_{\scT}])\!+\!(\hC\!+\!\hD\wh\Th_{\scT})\dbE[\wX_{\scT}]
          \!+\!\hD(\th_{\scT}\!+\!\wh\th_{\scT})\!+\!\si^*\big\}dW, \\
\wX_{\scT}(0) &= x-x^*.
\end{aligned}\right.\end{equation}
We now provide an estimate for $\dbE[\wX_{\scT}(t)]$.

\begin{proposition}\label{prop:EwX:bound}
Let {\rm\ref{MFTP:A1}--\ref{MFTP:A2}} hold. Then there exist constants $K,\l>0$, independent of $T$, such that
\begin{align}\label{EX:bound}
|\dbE[\wX_{\scT}(t)]| \les  K\[e^{-\l t}+e^{-\l(T-t)}\], \q\forall t\in[0,T].
\end{align}
\end{proposition}

\begin{proof}
The function $t\mapsto\dbE[\wX_{\scT}(t)]$ satisfies the following ODE:
\begin{equation}\label{state:cloop}\left\{\begin{aligned}
{d\over dt}\dbE[\wX_{\scT}(t)] &= [\hA+\hB\wh\Th_{\scT}(t)]\dbE[\wX_{\scT}(t)]+\hB[\th_{\scT}(t)+\wh\th_{\scT}(t)], \\
 \dbE[\wX_{\scT}(0)] &= x-x^* \equiv \tx.
\end{aligned}\right.\end{equation}
Using the notation
$$ \wh\cA_{\scT}(t) = \hA+\hB\wh\Th_{\scT}(t), \q \wh\Th= -\wh\cR(P)^{-1}\wh\cS(P,\vP), \q \wh\cA = \hA+\hB\wh\Th, $$
we have by the variation of constants formula that
\begin{align*}
\dbE[\wX_{\scT}(t)]
&= e^{\wh\cA t}\tx + \int_0^te^{\wh\cA(t-s)}\[\(\wh\cA_{\scT}(s)-\wh\cA\)\dbE[\wX_{\scT}(s)]
   + \hB\(\th_{\scT}(s)+\wh\th_{\scT}(s)\)\] ds.
\end{align*}
By \autoref{lmm:ARE-sol}, $\wh\cA$ is stable, so
$$ |e^{\wh\cA t}| \les Ke^{-\l t}, \q\forall t\ges0 $$
for some constants $K,\l>0$ independent of $T$. Since
\begin{align*}
\wh\Th - \wh\Th_{\scT}(t)
&= \(\hR+\hD^\top P_{\scT}(t)\hD\)^{-1}\(\hB^\top\vP_{\scT}(t)+\hD^\top P_{\scT}(t)\hC+\hS\)  \\
&\hp{=\ } - \(\hR+\hD^\top P\hD\)^{-1}\(\hB^\top\vP+\hD^\top P\hC+\hS\) \\
&= \(\hR+\hD^\top P_{\scT}(t)\hD\)^{-1}\(\hB^\top[\vP_{\scT}(t)-\vP]+\hD^\top[P_{\scT}(t)-P]\hC\)   \\
&\hp{=\ } + \[\(\hR+\hD^\top P_{\scT}(t)\hD\)^{-1}-\(\hR+\hD^\top P\hD\)^{-1}\]\(\hB^\top\vP+\hD^\top P\hC+\hS\)  \\
&= \(\hR+\hD^\top P_{\scT}(t)\hD\)^{-1}\(\hB^\top[\vP_{\scT}(t)-\vP]+\hD^\top[P_{\scT}(t)-P]\hC\)   \\
&\hp{=\ } +\(\hR+\hD^\top P_{\scT}(t)\hD\)^{-1}\hD^\top[P_{\scT}(t)-P]\hD\wh\Th,
\end{align*}
we have by \autoref{lmm:Si-P:bound} and \autoref{thm:vP-Pi:bound},
\begin{align}\label{hAT-hA}
|\wh\cA_{\scT}(t)-\wh\cA\,|
&\les K|\wh\Th_{\scT}(t)-\wh\Th| \les K\[|\vP_{\scT}(t)-\vP|+|P_{\scT}(t)-P|\] \nn\\
&\les Ke^{-\l(T-t)}, \q\forall t\in[0,T].
\end{align}
Also recalling \autoref{lmm:th:bound}, we obtain
\begin{align*}
|\dbE[\wX_{\scT}(t)]|
&\les Ke^{-\l t}  + K\int_0^t e^{-\l(t-s)}e^{-\l(T-s)}\[\dbE[\wX_{\scT}(s)]+1\] ds \\
&\les K\[e^{-\l t}+e^{-\l(T-t)}\] + K\int_0^t e^{-\l(T+t-2s)}\dbE[\wX_{\scT}(s)] ds \\
&\les K\[e^{-\l t}+e^{-\l(T-t)}\] + K\int_0^t e^{-\l(T-s)}\dbE[\wX_{\scT}(s)] ds.
\end{align*}
The desired result then follows from Gronwall's inequality.
\end{proof}

Recall the SDE \rf{SDE:X*}. We have the following result.

\begin{proposition}\label{prop:EX*bound}
Let {\rm\ref{MFTP:A1}--\ref{MFTP:A2}} hold. Then there exists a constant $K>0$ such that
the solution $X^*(\cd)$ to the SDE \rf{SDE:X*} satisfies
\begin{align}\label{X*:bound}
\dbE|X^*(t)|^2 \les K, \q\forall t\ges 0.
\end{align}
\end{proposition}

\begin{proof}
Let $P$ be the solution to the ARE \rf{ARE:P} and let $\Th= -\cR(P)^{-1}\cS(P)$. By It\^{o}'s rule,
\begin{align*}
&\dbE\lan PX^*(t),X^*(t)\ran \\
&\q= \dbE\int_0^t\Big\{2\lan PX^*(s),(A+B\Th)X^*(s)\ran \\
&\q\hp{=\ } +\lan P[(C+D\Th)X^*(s)+\si^*],(C+D\Th)X^*(s)+\si^*\ran\Big\}ds  \\
&\q= \dbE\int_0^t\Big\{\lan [P(A+B\Th)+(A+B\Th)^\top P + (C+D\Th)^\top P(C+D\Th)]X^*(s),X^*(s)\ran \\
&\q\hp{=\ } +2\lan P(C+D\Th)X^*(s),\si^*\ran + \lan P\si^*,\si^*\ran\Big\}ds  \\
&\q= \dbE\int_0^t\[-\lan(Q+\Th^\top R\Th + S^\top\Th +\Th^\top S)X^*(s),X^*(s)\ran \\
&\q\hp{=\ } + 2\lan P(C+D\Th)X^*(s),\si^*\ran + \lan P\si^*,\si^*\ran\]ds.
\end{align*}
Note that by \ref{MFTP:A1},
$$ Q+\Th^\top R\Th + S^\top\Th +\Th^\top S = Q-S^\top R^{-1}S + (\Th+R^{-1}S)^\top R(\Th+R^{-1}S) >0. $$
Let $\mu>0$ and $\nu>0$ be the smallest eigenvalues of $P$ and $Q+\Th^\top R\Th + S^\top\Th +\Th^\top S$, respectively. Then
\begin{align*}
\dbE|X^*(t)|^2
&\les {1\over\mu}\dbE\lan PX^*(t),X^*(t)\ran \\
&\les {1\over\mu}\dbE\int_0^t\lt[-\nu|X^*(s)|^2 + 2|X^*(s)|\cd|(C+D\Th)^\top P\si^*| + \lan P\si^*,\si^*\ran\rt]ds \\
&\les {1\over\mu}\dbE\int_0^t\lt[-{\nu\over2}|X^*(s)|^2 + {2\over\nu}|(C+D\Th)^\top P\si^*|^2 + \lan P\si^*,\si^*\ran\rt]ds.
\end{align*}
It follows from the Gronwall inequality that
\begin{align*}
\dbE|X^*(t)|^2
&\les {1\over\mu}\lt[{2\over\nu}|(C+D\Th)^\top P\si^*|^2 + \lan P\si^*,\si^*\ran\rt]\int_0^t\exp\lt[{\nu\over2\mu}(s-t)\rt]ds \\
&\les {2\over\nu}\lt[{2\over\nu}|(C+D\Th)^\top P\si^*|^2 + \lan P\si^*,\si^*\ran\rt].
\end{align*}
This completes the proof.
\end{proof}

We now prove \autoref{thm:main}.

\begin{proof}[\textbf{Proof of \autoref{thm:main}}]
Let $\Th_{\scT}(t) $ and $\wh\Th_{\scT}(t)$ be as in \rf{def:ThT} and \rf{def:wThT}, respectively, and let
$$\Th = -\cR(P)^{-1}\cS(P) , \q \wh\Th = -\wh\cR(P)^{-1}\wh\cS(P,\vP). $$
For notational simplicity, we write
\begin{align*}
&\cA= A+B\Th, \q \cC= C+D\Th, \q \cA_{\scT}(t)= A+B\Th_{\scT}(t), \q \cC_{\scT}(t)= C+D\Th_{\scT}(t), \\
&\wh\cA= \hA+\hB\wh\Th, \q \wh\cC= \hC+\hD\wh\Th, \q \wh\cA_{\scT}(t)= \hA+\hB\wh\Th_{\scT}(t), \q \wh\cC_{\scT}(t)= \hC+\hD\wh\Th_{\scT}(t).
\end{align*}
Then the process
$$ V_{\scT}(t)\deq \wX_{\scT}(t)-\dbE[\wX_{\scT}(t)]-X^*(t), \q t\in[0,T] $$
satisfies $V_{\scT}(0)=0$ and
\begin{align*}
dV_{\scT}(t)
&= \big\{\cA_{\scT}(t)V_{\scT}(t) + [\cA_{\scT}(t)-\cA]X^*(t) \big\}dt \\
&\hp{=\ } + \big\{\cC_{\scT}(t)V_{\scT}(t) + [\cC_{\scT}(t)-\cC]X^*(t) + \wh\cC_{\scT}(t)\dbE[\wX_{\scT}(t)]
          + \hD[\th_{\scT}(t)+\wh\th_{\scT}(t)]\big\}dW(t).
\end{align*}
Let $P>0$ be the solution to the ARE \rf{ARE:P}. Then It\^{o}'s rule implies that
\begin{align}\label{EV=}
&\dbE\lan PV_{\scT}(t),V_{\scT}(t)\ran
= \dbE\int_0^t\Big\{2\lan PV_{\scT},\cA_{\scT}V_{\scT} + (\cA_{\scT}-\cA)X^*\ran \nn\\
&\qq\qq +\lan P[\cC_{\scT}V_{\scT} +(\cC_{\scT}-\cC)X^* +h_{\scT}],\cC_{\scT}V_{\scT} +(\cC_{\scT}-\cC)X^* +h_{\scT}\ran\Big\}ds,
\end{align}
where we have suppressed $s$ in the integrand, and
\begin{align*}
h_{\scT}(s) &= \wh\cC_{\scT}(s)\dbE[\wX_{\scT}(s)] + \hD[\th_{\scT}(s)+\wh\th_{\scT}(s)].
\end{align*}
Note that
$$ \dbE[X^*(t)] = 0, \q \dbE[V_{\scT}(t)] = 0, \q\forall t\in[0,T]. $$
Thus,
\begin{align*}
&\dbE\lan P[\cC_{\scT}V_{\scT}+(\cC_{\scT}-\cC)X^*+h_{\scT}],\cC_{\scT}V_{\scT}+(\cC_{\scT}-\cC)X^*+h_{\scT}\ran \\
&\q= \dbE\lan P[\cC_{\scT}V_{\scT}+(\cC_{\scT}-\cC)X^*],\cC_{\scT}V_{\scT}+(\cC_{\scT}-\cC)X^*\ran + \lan Ph_{\scT},h_{\scT}\ran.
\end{align*}
Substituting the above into \rf{EV=} and expanding the integrand yields
\begin{align}\label{EV=2}
\dbE\lan PV_{\scT}(t),V_{\scT}(t)\ran
&=\int_0^t\dbE\Big\{\blan (P\cA_{\scT}+ \cA_{\scT}^\top P + \cC_{\scT}^\top P\cC_{\scT})V_{\scT},V_{\scT}\bran
 +\blan\wh\cC_{\scT}^\top P\wh\cC_{\scT}\dbE[\wX_{\scT}],\dbE[\wX_{\scT}]\bran \nn\\
&\hp{=\ } +2\blan[P(\cA_{\scT}-\cA)+\cC_{\scT}^\top P(\cC_{\scT}-\cC)]X^*,V_{\scT}\bran + k_{\scT}\Big\}ds,
\end{align}
where
\begin{align*}
k_{\scT}(s) &= \blan P[\cC_{\scT}(s)-\cC]X^*(s),[\cC_{\scT}(s)-\cC]X^*(s)\bran \\
&\hp{=\ } +\blan 2P\wh\cC_{\scT}(s)\dbE[\wX_{\scT}(s)] + P\hD[\th_{\scT}(s)+\wh\th_{\scT}(s)],\hD[\th_{\scT}(s)+\wh\th_{\scT}(s)]\bran.
\end{align*}
Next, Let $\vP>0$ be the solution to the ARE \rf{ARE:Pi}. Then integration by parts gives
\begin{align}\label{vPEwX=}
&\blan\vP\dbE[\wX_{\scT}(t)],\dbE[\wX_{\scT}(t)]\bran - \lan\vP\tx,\tx\ran  \nn\\
&\q= \int_0^t \Big\{\blan(\vP\wh\cA_{\scT}+\wh\cA_{\scT}^\top\vP)\dbE[\wX_{\scT}],\dbE[\wX_{\scT}]\bran
     +2\blan\vP\dbE[\wX_{\scT}],\hB[\th_{\scT}+\wh\th_{\scT}]\bran\Big\}ds.
\end{align}
Adding \rf{EV=2}  and \rf{vPEwX=}, we obtain
\begin{align}\label{EV+vPEwX=}
&\dbE\lan PV_{\scT}(t),V_{\scT}(t)\ran + \blan\vP\dbE[\wX_{\scT}(t)],\dbE[\wX_{\scT}(t)]\bran - \lan\vP\tx,\tx\ran \nn\\
&\q=\int_0^t\dbE\Big\{\blan (P\cA_{\scT}+ \cA_{\scT}^\top P + \cC_{\scT}^\top P\cC_{\scT})V_{\scT},V_{\scT}\bran \nn\\
&\q\hp{=\ } +\blan(\vP\wh\cA_{\scT}+\wh\cA_{\scT}^\top\vP+\wh\cC_{\scT}^\top P\wh\cC_{\scT})\dbE[\wX_{\scT}],\dbE[\wX_{\scT}]\bran \nn\\
&\q\hp{=\ } +2\blan[P(\cA_{\scT}-\cA)+\cC_{\scT}^\top P(\cC_{\scT}-\cC)]X^*,V_{\scT}\bran + \phi_{\scT}\Big\}ds,
\end{align}
where
$$ \phi_{\scT}(s)=k_{\scT}(s)+2\blan\vP\dbE[\wX_{\scT}(s)],\hB[\th_{\scT}(s)+\wh\th_{\scT}(s)]\bran. $$
We observe the following facts:
\begin{align*}
& P\cA_{\scT}(s) + \cA_{\scT}(s)^\top P + \cC_{\scT}(s)^\top P\cC_{\scT}(s) \\
&\q= P\cA + \cA^\top P + \cC^\top P\cC + P[\cA_{\scT}(s)-\cA] + [\cA_{\scT}(s)-\cA]^\top P \\
&\q\hp{=\ } + [\cC_{\scT}(s)-\cC]^\top P\cC + \cC_{\scT}(s)^\top P[\cC_{\scT}(s)-\cC], \\[2mm]
& P\cA + \cA^\top P + \cC^\top P\cC = -(Q+\Th^\top R \Th + S^\top\Th  +\Th^\top S) <0, \\[2mm]
& \vP\wh\cA_{\scT}(s)+\wh\cA_{\scT}(s)^\top\vP+\wh\cC_{\scT}(s)^\top P\wh\cC_{\scT}(s) \\
&\q= \vP\wh\cA+\wh\cA^\top\vP+\wh\cC^\top P\wh\cC + \vP[\wh\cA_{\scT}(s)-\wh\cA\,] + [\wh\cA_{\scT}(s)-\wh\cA\,]^\top\vP \\
&\q\hp{=\ } + [\wh\cC_{\scT}(s)-\wh\cC\,]^\top P\wh\cC + \wh\cC_{\scT}(s)^\top P[\wh\cC_{\scT}(s)-\wh\cC\,], \\[2mm]
& \vP\wh\cA+\wh\cA^\top\vP+\wh\cC^\top P\wh\cC = -(\hQ+\wh\Th^\top\hR\wh\Th+\hS^\top\wh\Th+\wh\Th^\top\hS\,) <0.
\end{align*}
Also, observe that for some constant $K>0$,
$$|\cA_{\scT}(t)|+|\wh\cA_{\scT}(t)|+|\cC_{\scT}(t)|+|\wh\cC_{\scT}(t)| \les K, \q\forall 0\les t\les T<\i, $$
and similar to \rf{hAT-hA}, we have
$$|\cA_{\scT}(t)-\cA| +|\wh\cA_{\scT}(t)-\wh\cA\,| + |\cC_{\scT}(t)-\cC|
+ |\wh\cC_{\scT}(t)-\wh\cC\,|\les Ke^{-\l(T-t)}, \q\forall 0\les t\les T<\i, $$
for some constants $K,\l>0$.
Then it follows that for some constant $\a>0$,
\begin{align}
&\dbE\blan[P\cA_{\scT}(s)+ \cA_{\scT}(s)^\top P + \cC_{\scT}(s)^\top P\cC_{\scT}(s)]V_{\scT}(s),V_{\scT}(s)\bran \nn\\
&\q\les K\[-2\a + e^{-\l(T-s)}\]\dbE|V_{\scT}(s)|^2, \label{09-03:1}\\[2mm]
&\blan[\vP\wh\cA_{\scT}(s)+\wh\cA_{\scT}(s)^\top\vP+\wh\cC_{\scT}(s)^\top P\wh\cC_{\scT}(s)]\dbE[\wX_{\scT}(s)],\dbE[\wX_{\scT}(s)]\bran \nn\\
&\q\les  K\[-2\a + e^{-\l(T-s)}\]|\dbE[\wX_{\scT}(s)]|^2. \nn
\end{align}
Using \autoref{prop:EwX:bound}, we can further conclude that
\begin{align}\label{09-03:2}
\blan[\vP\wh\cA_{\scT}(s)+\wh\cA_{\scT}(s)^\top\vP+\wh\cC_{\scT}(s)^\top P\wh\cC_{\scT}(s)]\dbE[\wX_{\scT}(s)],\dbE[\wX_{\scT}(s)]\bran
\les  K e^{-\l(T-s)}.
\end{align}
Moreover, by the Cauchy--Schwarz inequality and \rf{X*:bound},
\begin{align}\label{09-03:3}
& 2\dbE\blan[P(\cA_{\scT}(s)-\cA)+\cC_{\scT}(s)^\top P(\cC_{\scT}(s)-\cC)]X^*(s),V_{\scT}(s)\bran \nn\\
&\q\les K\[\a\dbE|V_{\scT}(s)|^2 + e^{-\l(T-s)}\],
\end{align}
and by \autoref{lmm:th:bound}, \autoref{prop:EwX:bound}, and \autoref{prop:EX*bound},
\begin{align}\label{09-03:4}
 \dbE|\phi_{\scT}(s)| \les  K e^{-\l(T-s)}.
\end{align}
Substitution of \rf{09-03:1}--\rf{09-03:4} into \rf{EV+vPEwX=}, we obtain
\begin{align*}
&\dbE\lan PV_{\scT}(t),V_{\scT}(t)\ran- \lan\vP\tx,\tx\ran \\
&\q\les \dbE\lan PV_{\scT}(t),V_{\scT}(t)\ran + \blan\vP\dbE[\wX_{\scT}(t)],\dbE[\wX_{\scT}(t)]\bran - \lan\vP\tx,\tx\ran \\
&\q\les  K\int_0^t\[\(-\a + e^{-\l(T-s)}\)\dbE|V_{\scT}(s)|^2 + e^{-\l(T-s)}\]ds.
\end{align*}
Since $P>0$, the above implies that
\begin{align*}
\dbE|V_{\scT}(t)|^2 \les K + K\int_0^t\[\(-\a + e^{-\l(T-s)}\)\dbE|V_{\scT}(s)|^2 + e^{-\l(T-s)}\]ds.
\end{align*}
Then, by Gronwall's inequality, we have
\begin{align*}
\dbE|V_{\scT}(t)|^2 \les  K\[e^{-\l t}+e^{-\l(T-t)}\], \q\forall t\in[0,T],
\end{align*}
with possibly different constants $K,\l>0$. The desired \rf{X*-X*T} now follows from
\begin{align*}
X^*_{\scT}(t)-\bm X^*(t) = V_{\scT}(t)+\dbE[\wX_{\scT}(t)]
\end{align*}
and
\begin{align*}
u^*_{\scT}(t)-\bm u^*(t)
&= \Th_{\scT}(t)[X^*_{\scT}(t)-\bm X^*(t)] + [\Th_{\scT}(t)-\Th]X^*(t) \\
&\hp{=\ } +[\wh\Th_{\scT}(t)-\Th_{\scT}(t)]\dbE[\wX_{\scT}(t)] + \th_{\scT}(t) + \wh\th_{\scT}(t).
\end{align*}
The proof is complete.
\end{proof}

Next, we prove \autoref{crlry:value}.

\begin{proof}[\textbf{Proof of \autoref{crlry:value}}]
By \autoref{prop:EX*bound},
$$ \dbE|\bm X^*(t)|^2 + \dbE|\bm u^*(t)|^2\les K, \q\forall t\ges0 $$
for some constant $K>0$, and hence by \autoref{thm:main},
$$ \dbE|X^*_{\scT}(t)|^2 + \dbE|u^*_{\scT}(t)|^2 \les K, \q\forall 0\les t\les T<\i, $$
for some possibly different constant $K>0$. The above, together with \autoref{crlry:jifen-TP},
implies that
\begin{align*}
&\lim_{T\to\i} {1\over T}V_{\scT}(x) = \lim_{T\to\i} {1\over T}J_{\scT}(x;u^*_{\scT}(\cd)) \\
&\q= \lim_{T\to\i} {1\over T}\int_0^T\Big\{\dbE\big[\lan QX^*_{\scT}(t),X^*_{\scT}(t)\ran
   +2\lan SX^*_{\scT}(t),u^*_{\scT}(t)\ran + \lan Ru^*_{\scT}(t),u^*_{\scT}(t)\ran\big] \\
&\q\hp{=\ } +\lan\bQ\dbE[X^*_{\scT}(t)],\dbE[X^*_{\scT}(t)]\ran + 2\lan\bS\dbE[X^*_{\scT}(t)],\dbE[u^*_{\scT}(t)]\ran
          +\lan\bR\dbE[u^*_{\scT}(t)],\dbE[u^*_{\scT}(t)]\ran \\
&\q\hp{=\ } +2\lan q,\dbE[X^*_{\scT}(t)]\ran + 2\lan r,\dbE[u^*_{\scT}(t)]\ran \Big\}dt \\
&\q= \lim_{T\to\i} {1\over T}\int_0^T\Big\{\dbE\big[\lan Q\bm X^*(t),\bm X^*(t)\ran
     +2\lan S\bm X^*(t),\bm u^*(t)\ran + \lan R\bm u^*(t),\bm u^*(t)\ran\big] \\
&\q\hp{=\ } +\lan\bQ\dbE[\bm X^*(t)],\dbE[\bm X^*(t)]\ran + 2\lan\bS\dbE[\bm X^*(t)],\dbE[\bm u^*(t)]\ran
            +\lan\bR\dbE[\bm u^*(t)],\dbE[\bm u^*(t)]\ran  \\
&\q\hp{=\ } +2\lan q,\dbE[\bm X^*(t)]\ran + 2\lan r,\dbE[\bm u^*(t)]\ran \Big\}dt.
\end{align*}
Noting that
$$ \dbE[\bm X^*(t)] = \dbE[X^*(t) + x^*] = x^*, \q \dbE[\bm u^*(t)]= \dbE[\Th X^*(t) + u^*] = u^*,$$
we further have
\begin{align}\label{VT:limit-1}
\lim_{T\to\i} {1\over T}V_{\scT}(x)
&=\lim_{T\to\i} {1\over T}\int_0^T\Big\{\dbE\big[\lan QX^*(t),X^*(t)\ran  + \lan Q x^*,x^*\ran \nn\\
&\hp{=\ } +2\lan SX^*(t),\Th X^*(t)\ran + 2\lan Sx^*,u^*\ran  \nn\\
&\hp{=\ } +\lan R\Th X^*(t),\Th X^*(t)\ran+ \lan R u^*,u^*\ran \big] \nn\\
&\hp{=\ } +\lan\bQ x^*,x^*\ran + 2\lan\bS x^*,u^*\ran + \lan\bR u^*,u^*\ran + 2\lan q,x^*\ran + 2\lan r,u^*\ran \Big\}dt \nn\\
&=\lim_{T\to\i} {1\over T}\int_0^T \dbE\lan (Q+S^\top\Th+\Th^\top S +\Th^\top R\Th)X^*(t),X^*(t)\ran dt \nn\\
&\hp{=\ } +\lan\hQ x^*,x^*\ran + 2\lan\hS x^*,u^*\ran + \lan\hR u^*,u^*\ran + 2\lan q,x^*\ran + 2\lan r,u^*\ran.
\end{align}
On the other hand, with $P>0$ being the solution to the ARE \rf{ARE:P} and noting that $\dbE[X^*(t)]=0$, we have
\begin{align*}
\dbE\lan PX^*(T),X^*(T)\ran
&= \dbE\int_0^T \Big\{2\lan PX^*(t),(A+B\Th)X^*(t)\ran \\
&\hp{=\ } + \lan P[(C+D\Th)X^*(t)+\si^*],(C+D\Th)X^*(t)+\si^*\ran\Big\}dt \\
&= \dbE\int_0^T \Big\{\lan [P(A+B\Th)+(A+B\Th)^\top P \\
&\hp{=\ } + (C+D\Th)^\top P(C+D\Th)]X^*(t),X^*(t)\ran + \lan P\si^*,\si^*\ran\Big\}dt \\
&= \dbE\int_0^T \!\[\!-\lan(Q+S^\top\!\Th+\Th^\top\! S +\Th^\top\! R\Th)X^*(t),X^*(t)\ran + \lan P\si^*,\si^*\ran\]dt.
\end{align*}
Since $\dbE|X^*(T)|^2$ is bounded in $T$ (see \autoref{prop:EX*bound}), we obtain
\begin{align}\label{VT:limit-2}
&\lim_{T\to\i} {1\over T}\int_0^T \dbE\lan(Q+S^\top\Th+\Th^\top S +\Th^\top R\Th)X^*(t),X^*(t)\ran dt \nn\\
&\q=\lim_{T\to\i} {1\over T}\lt[-\dbE\lan PX^*(T),X^*(T)\ran + \int_0^T\lan P\si^*,\si^*\ran dt\rt] \nn\\
&\q= \lan P\si^*,\si^*\ran.
\end{align}
Combining \rf{VT:limit-1} and \rf{VT:limit-2}, we get the desired result.
\end{proof}

Finally, we prove \autoref{crlry:TP-adjoint}.

\begin{proof}[\textbf{Proof of \autoref{crlry:TP-adjoint}}]
By a tedious but straightforward calculation we can verify that in \rf{os:wX},
the adapted solution $(\wY_{\scT}(\cd),Z^*_{\scT}(\cd))$ to the backward SDE
has the following relation with $(\wX_{\scT}(\cd),\tu_{\scT}(\cd))$:
\begin{align*}
\wY_{\scT}(t) &= P_{\scT}(t)\Big\{\wX_{\scT}(t)-\dbE[\wX_{\scT}(t)]\Big\} + \vP_{\scT}(t)\dbE[\wX_{\scT}(t)] + \wh\f_{\scT}(t), \\
Z^*_{\scT}(t) &= P_{\scT}(t)\Big\{C\wX_{\scT}(t) + D\tu_{\scT}(t) + \bC\dbE[\wX_{\scT}(t)] + \bD\dbE[\tu_{\scT}(t)] + \si^*\Big\}.
\end{align*}
Using the relation \rf{rep:hu}, we further obtain
\begin{align*}
Z^*_{\scT}(t) &= P_{\scT}(t)\Big\{[C+D\Th_{\scT}(t)]\wX_{\scT}(t)
                 + [\bC-D\Th_{\scT}(t)]\dbE[\wX_{\scT}(t)] + \hD\dbE[\tu_{\scT}(t)] + \si^*\Big\}.
\end{align*}
For $\dbE|Y^*_{\scT}(t)-\bm Y^*(t)|^2$, we first observe that
\begin{align*}
Y^*_{\scT}(t)-\bm Y^*(t)
&= \wY_{\scT}(t)- PX^*(t) \\
&= P_{\scT}(t)\wX_{\scT}(t)-PX^*(t)+ [\vP_{\scT}(t)-P_{\scT}(t)]\dbE[\wX_{\scT}(t)] + \wh\f_{\scT}(t)  \\
&= P_{\scT}(t)[X^*_{\scT}(t)-\bm X^*(t)] + [P_{\scT}(t)-P]X^*(t) \\
&\hp{=\ } + [\vP_{\scT}(t)-P_{\scT}(t)]\dbE[\wX_{\scT}(t)] + \wh\f_{\scT}(t).
\end{align*}
By \autoref{lmm:Si-P:bound} and \autoref{thm:vP-Pi:bound}, $P_{\scT}(\cd)$ and $\vP_{\scT}(\cd)$ are bounded uniformly in $T$,
and by \autoref{prop:EX*bound}, $\dbE|X^*(\cd)|$ is also bounded. The above then implies that
\begin{align*}
\dbE|Y^*_{\scT}(t)-\bm Y^*(t)|^2
&\les K\Big\{\dbE|X^*_{\scT}(t)-\bm X^*(t)|^2 + |P_{\scT}(t)-P|^2 + |\dbE[\wX_{\scT}(t)]|^2 + |\wh\f_{\scT}(t)|^2\Big\},
\end{align*}
for some constant $K>0$. Recalling \autoref{lmm:Si-P:bound}, \autoref{thm:main}, \autoref{lmm:th:bound},
and \autoref{prop:EwX:bound}, we obtain that for some constants $K,\l>0$ independent of $T$,
\begin{align*}
\dbE|Y^*_{\scT}(t)-\bm Y^*(t)|^2 \les K\[e^{-\l t} + e^{-\l(T-t)}\], \q\forall t\in[0,T].
\end{align*}
In a similar manner, we can show that
\begin{align*}
\dbE|Z^*_{\scT}(t)-\bm Z^*(t)|^2 \les K\[e^{-\l t} + e^{-\l(T-t)}\], \q\forall t\in[0,T].
\end{align*}
The proof is complete.
\end{proof}

\end{document}